%% file: main.tex
\documentclass[11pt]{amsart}
\usepackage[margin=1.5in]{geometry}
\usepackage{amsmath}
\usepackage{stackengine}
\usepackage{amsthm}
\usepackage{blkarray}
\usepackage{mathrsfs}
\usepackage{amssymb}
\usepackage{mathtools}
\usepackage{xcolor}
\usepackage{algorithm}
\usepackage{algorithmic}
\usepackage{bbold}
\usepackage{cite}
\usepackage{latexsym}
\usepackage{booktabs}
\usepackage{bm}
\usepackage{comment}
\usepackage{graphicx}
\usepackage{textcomp}
\usepackage{subfigure}
\usepackage[toc,page]{appendix}
\usepackage{url}
\usepackage{hyperref}
\usepackage{multicol}
\usepackage{multirow}
 \usepackage{todonotes}
\usepackage{tikz}
\usetikzlibrary{3d}
\usepackage{float}
\usepackage{fancyhdr} 
\let\oldtextit\textit 
\renewcommand\textit[1]{\oldtextit{\color{RoyalBlue}#1}}
\definecolor{RoyalBlue}{cmyk}{1, 0.50, 0, 0}
\definecolor{Amber}{rgb}{1.0, 0.49, 0.0}
\definecolor{zzwwqq}{rgb}{0.6,0.4,0}

\theoremstyle{definition}
\newtheorem{definition}{Definition}[section]

\newtheorem{proposition}[definition]{Proposition}
\newtheorem{prop}[definition]{Proposition}

\newtheorem{lemma}[definition]{Lemma}

\newtheorem{theorem}[definition]{Theorem}

\newtheorem{example}[definition]{Example}

\newtheorem{conj}[definition]{Conjecture}
\newtheorem{construction}[definition]{Construction}
\theoremstyle{remark}
\newtheorem{remark}[definition]{Remark}

\newcommand{\CC}{\mathbb C}
\newcommand{\RR}{\mathbb R}
\newcommand{\QQ}{\mathbb Q}
\newcommand{\ZZ}{\mathbb Z}

\newcommand{\KK}{\mathbb K}

\newcommand{\cT}{\mathcal T}

\newcommand{\kk}{\mathbf k}

\makeatletter
\def\@settitle{\begin{center}%
  \baselineskip13\p@\relax
    \Large
\@title
  \end{center}%
}
\makeatother
\title[Symmetric Tropical Rank 2 Matrices]{The Tropical Variety of Symmetric Rank 2 Matrices}

\author{May Cai}
\address{School of Mathematics, Georgia Tech, Atlanta GA 30332, US}
\email{mcai@gatech.edu}

\author{Kisun Lee}
\address{School of Mathematical and Statistical Science, Clemson University, Clemson, SC, 29634, US}
\email{kisunl@clemson.edu}

\author{Josephine Yu}
\address{School of Mathematics, Georgia Tech, Atlanta GA 30332, US}
\email{jyu@math.gatech.edu}

\date{}

\begin{document}

\maketitle
\begin{abstract}
    We study the tropicalization of the variety of symmetric rank two matrices.  Analogously to the result of Markwig and Yu for general tropical rank two matrices, we show that it has a simplicial complex structure as the space of symmetric bicolored trees and that this simplicial complex is shellable.  We also discuss some matroid structures arising from this space and present generating functions for the number of symmetric bicolored trees.
\end{abstract}

\section{Introduction}

A tropical matrix is a matrix with entries in the $(\min,+)$ tropical semiring.  We will use $\oplus = \min$ and $\odot = +$ to denote the tropical arithmetic operations.   There are a few different notions of the \textit{rank} of a tropical matrix $A$, including the tropical rank, the Kapranov rank over a valued field, and the Barvinok rank.
For rank $1$, all these notions coincide.  In general they can all be different, but a tropical matrix has tropical rank $2$ if and only if it has Kapranov rank $2$ \cite{DevelinSantosSturmfels}.

Markwig and Yu showed that the space of $d \times n$ matrices of tropical and Kapranov rank $2$ has a simplicial fan structure as the space of \textit{bicolored trees}, and that it is shellable \cite{MarkwigYu}.
Motivated by matrix completion problems, Bernstein described the algebraic matroid for the variety of $d\times n$ rank $2$ matrices using tropical geometry~\cite{bernstein2017completion}.
In this paper, we study analogues of these results for symmetric tropical rank $2$ matrices. 

For symmetric tropical matrices, Cartwright and Chan studied three notions of ranks, namely the symmetric Barvinok rank, the star tree rank, and the tree rank, which depend respectively on how a  symmetric tropical matrix decomposes as a tropical sum of rank one  symmetric tropical matrices, star tree dissimilarity matrices, or tree dissimilarity matrices~\cite{cartwright2012three}.  A  symmetric tropical matrix has finite symmetric Barvinok rank if and only if it is a tropical positive semidefinite matrix~\cite{Yu_PSD}.  
Zwick introduced the symmetric tropical rank and the  symmetric Kapranov rank~\cite{zwick2014variations, zwick2021symmetric} and showed that they coincide for rank $2$ symmetric matrices for the Kapranov rank over an algebraically closed valued field of characteristic $0$.  The symmetric Kapranov rank arises from tropicalizing the variety of symmetric rank $2$ matrices, so it is compatible with the symmetric matrix completion problem as in~\cite{bernstein2020typical,bernstein2021typical}.

Our main results are a combinatorial description of a simplicial fan structure for the space of symmetric tropical rank~2 matrices and a proof that it is shellable.  A shelling is a total ordering of maximal cells in a simplicial complex such that each cell intersects the union of preceding cells in codimension one. Tropical varieties which are known to be shellable include Bergman fans~\cite{ArdilaKlivans}, stable intersections of tropical hypersurfaces with themselves which coincide with skeleta of polytopal fans, tropical Grassmannians $\mathop{Gr}(2,n)$ which coincide with the spaces of phylogenetic trees~\cite{trappmann1998shellability}, and the space of rank $2$ matrices mentioned above.   Shellability implies that the simplicial complex is homotopy equivalent to a wedge of spheres, and in particular has homology only in the top dimension.  Hacking gave a criterion for the link of a tropical variety fan to have homology only in the top dimension and also gave some  examples that do not arise from shellability results~\cite{Hacking}.

In Section~\ref{sec:symbic} we show that symmetric tropical rank~2 matrices are indexed by \textit{sym}metric \textit{bic}olored trees on $2n$ leaves (\textit{symbic trees} for short), and we give two parameterizations of the polyhedral cones using paths in the tree in Section~\ref{sec:Cones}. We show in Section~\ref{sec:Shellability} that the simplicial complex of symbic trees is shellable.  We also gives a generating function for the number of regular symbic trees in Section~\ref{sec:enumeration} and discuss the algebraic matroid structure of the cones in the tropical variety in Section~\ref{sec:Algebraic Matroids}.

\section{Symmetric bicolored trees}
\label{sec:symbic}

A \textit{nonarchimedean valuation} of a field $\KK$ is a map  $\nu: \KK\setminus \{0\} \rightarrow \RR$ satisfying
\[\nu(ab) = \nu(a) \odot \nu(b) := \nu(a) + \nu(b) \text{ and } \nu(a+b) \geq \nu(a) \oplus \nu(b) := \min(\nu(a), \nu(b)) \]
for all $a,b \in \KK \setminus \{0\}$.  For example, the map that sends everything to zero is the trivial nonarchimedean valuation.  Other examples include $\QQ$ or $\CC_p$ with $p$-adic valuations, and the fields of Laurent series or Puiseux series where the valuation sends a formal power series to the smallest exponent appearing with a nonzero coefficient.

For a subvariety $V \subset (\KK^*)^n$ where $\KK$ is an algebraically closed field with nontrivial nonarchimedean valuation, 
we define the \textit{tropicalization} of $V$ as
\[\text{trop}(V)=\overline{\left\{(\nu(x_1),\dots, \nu(x_n)\mid (x_1,\dots, x_n)\in V\right\}}\subset \mathbb{R}^n.\]

When the variety $V=\mathbf{V}(I)$ is defined by some ideal $I \subset \kk[x_1,\dots, x_n]$  over a field $\kk$ with trivial valuation, we can first extend $\kk$ to an algebraically closed field $\KK \supset \kk$ with nontrivial valuation.  Then we extend $I$ to $I' \subseteq \KK[x_1,\dots, x_n]$ and define the variety $V'=\mathbf{V}(I')\subset (\KK^*)^n$. Then, we define $\text{trop}(V)=\text{trop}(V')$.  It follows from the {\em Fundamental Theorem} of tropical algebraic geometry that this definition does not depend on the choice of the extension $\KK$.
The {\em Structure Theorem} says that if $V$ is a irreducible variety of dimension $d$, then $\text{trop}(V)$ is the support of a balanced polyhedral complex of pure dimension $d$ which is connected through codimension one.  See Chapters 2 and 3 of \cite{maclagan2021introduction}.

For a polynomial over a field with a nonarchimedean valuation, its \textit{tropicalization} is obtained by replacing usual sums and products with tropical sums and products, while replacing the coefficients with their valuations.
For example, consider the determinant of a $3 \times 3$ symmetric matrix of variables
$\begin{bmatrix}x_{11}&x_{12}&x_{13}\\x_{12}&x_{22}&x_{23}\\x_{13}&x_{23}&x_{33}\end{bmatrix}$, which is $x_{11}x_{22}x_{33} + 2 x_{12}x_{13}x_{23} - x_{11}x_{23}^2 - x_{22} x_{13}^2 - x_{33}x_{12}^2$.
The tropicalized $3\times 3$ minor is the tropical polynomial \[ (x_{11}\odot x_{22} \odot x_{33}) \oplus (x_{12} \odot x_{23} \odot x_{13}) \oplus (x_{11} \odot x_{23} \odot x_{23}) \oplus (x_{22}  \odot x_{13} \odot x_{13}) \oplus (x_{33}\odot  x_{12} \odot x_{12}).\]

A symmetric $n \times n$ matrix has \textit{symmetric tropical rank} $r$ if the minimum is attained twice in all $(r+1)\times (r+1)$ minors where the minors are considered as polynomials in ${n+1 \choose 2}$ variables $x_{ij}$ for $1 \leq i \leq j \leq n$. 
For example, for the matrix $\begin{bmatrix}1&0&0\\0&1&0\\0&0&1\end{bmatrix}$ the minimum in the tropical determinant is attained twice, but the minimum is attained uniquely at the monomial $x_{12} \odot x_{23} \odot x_{13}$ in the symmetric tropical determinant.  
Thus $I$ has symmetric tropical rank $3$, although it has usual tropical rank $2$.  Zwick showed that a symmetric matrix has symmetric tropical rank $2$ if and only if it is in the tropicalization of the  variety of symmetric rank $2$ matrices~\cite{zwick2014variations, zwick2021symmetric}.  In particular, a symmetric matrix over rational numbers has symmetric tropical rank~$\leq 2$ if and only if it is the valuation of a symmetric matrix of rank~$\leq 2$ over the complex Puiseux series.

We will now recall the construction in~\cite{MarkwigYu, Develin_moduli} that associates a bicolored tree to a tropical rank $2$ matrix. 

\begin{construction}\label{con:tree-from-matrix}
Given a $r \times n$ tropical matrix $A$ of tropical rank $2$, we can construct a metric tree $T_A$ on
$2n$ leaves with labels $1,2,\dots,n, 1',2',\dots, r'$ as follows. 
The tropical convex hull of the columns of $A$, defined as all tropical linear combinations of the columns of $A$, is a polyhedral complex of dimension one modulo tropical scaling, that is, it is a union of line segments in $\RR^r/\RR(1,\dots,1)$.  The tropical convex hull is contractible, so it can be seen as a metric tree where the nodes are the points in the tropical convex hull  whose neighborhood is not an open line segment, and two nodes are connected by an edge if there is a straight line segment connecting them in the tropical convex hull.  The length of an edge is the length of the line segment under the tropical Hilbert metric
\[
d(x,y) = \max_{i} (x_i-y_i) - \min_{i} (x_i-y_i).
\]
For a tropical line to be balanced, the lattice direction of edges emanating from every vertex should sum up to a vector in $\RR(1,\dots,1)$. The tropical convex hull can be made balanced by attaching one infinite ray in each of the coordinate directions $\mathbf{e}_1, \dots, \mathbf{e}_n$ in a unique way, making it into a tropical line (see \cite[Figure 1]{Develin_moduli}).  
The desired tree $T_A$ is obtained by starting with the tropical convex hull as a metric tree, and attaching a leaf labeled $i$ to the location of the $i$-th column vector and a leaf labeled $j'$ to where the infinite ray $\mathbf{e}_j$ is attached in the tropical convex hull. 

The edges of the metric tree obtained from the tropical convex hull are called \textit{internal edges}. The remaining edges are called \textit{leaf edges} and are considered to have infinite length.
The two types of leaves are said to have different ``colors.''
\end{construction}

Tropically scaling a column of $A$, i.e.\ adding the same real number to all the entries in a column, does not change the tropical convex hull of $A$, so that does not change the tree $T_A$.  Tropically scaling a row of $A$ translates the tropical convex hull of $A$ in a coordinate direction and also does not change the tree $T_A$.

A \textit{split} of a tree is a partition of the leaf labels given by removing an internal non-leaf edge.
It was shown in \cite{MarkwigYu} that every split of such a tree constructed above has both colors on both sides.   We will call such trees \textit{bicolored} trees. Conversely, every bicolored metric tree uniquely determines a matrix of tropical rank $2$ modulo tropically scaling rows and columns.

\begin{example}
\label{ex:identity}
The bicolored trees corresponding to the following two matrices are depicted in Figure~\ref{fig:3x3}:
\[\begin{bmatrix}1&0&0\\ 0&1&0 \\ 0&0&1 \end{bmatrix} \text{ and }
 \begin{bmatrix}1&0&0\\ 0&0&1 \\ 0&1&0 \end{bmatrix}.\]
 The classical identity matrix, on the left, has symmetric tropical rank $3$.  The fixed points under the color changing involution form a subtree that is not a path.  The matrix on the right has symmetric tropical rank $2$.  The fixed point set consists of just one edge.
\end{example}

\begin{figure}
\centering
\input{figures/01symbictrees}
   \caption{The embedding of the matrices in Example~\ref{ex:identity}. The left is the embedding of the $3 \times 3$ identity matrix. All the internal edges are fixed points of the involution from Definition~\ref{def:symbic} on the left tree, while on the right the single fixed point internal edge is thicker.}
\label{fig:3x3}
\end{figure}
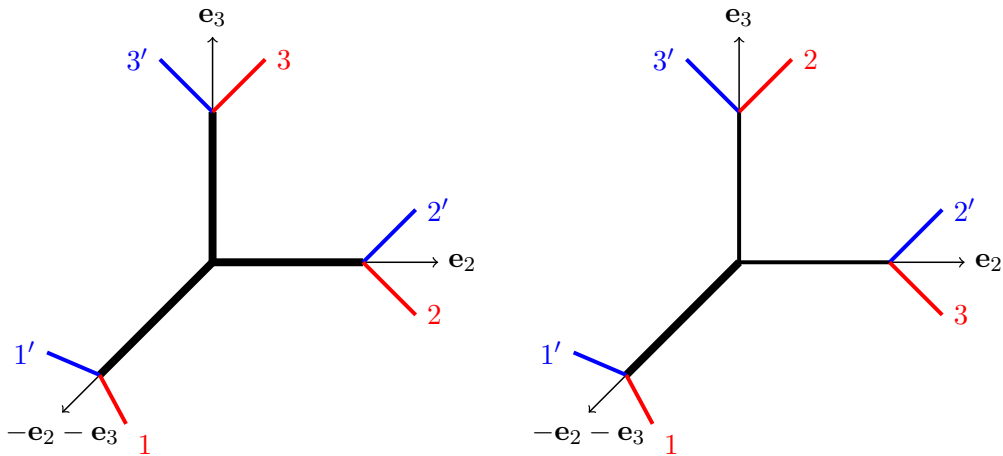

\begin{definition}\label{def:symbic}
   A bicolored tree is called a \textit{symmetric bicolored tree}  (\textit{symbic tree} for short) if it is a metric tree on leaves labeled $1,2,\dots,n, 1',2',\dots,n'$ such that 
   \begin{enumerate}
       \item each split contains both colors on both sides
       \item the internal edges have positive lengths
       \item the metric tree is symmetric with respect to the involution swapping  $i$ and $i'$ for all $i =1,2,\dots,n$
       \item the fixed points of the involution form a path.
    \end{enumerate}
    We do not consider the lengths of the leaf edges.  A \textit{combinatorial symbic tree} is a symbic tree without the metric data.  We will often drop the word combinatorial if it is clear from context.
    We will often use the shorthand ``$n + n$ symbic tree'' for a symbic tree on $2n$ leaves with $n$ leaves of each color, as they are associated by the following theorem with $n \times n$ matrices. 
\end{definition}

\begin{theorem}
\label{thm:symbic}
Let $A$ be a symmetric tropical matrix of symmetric tropical rank $2$.  Then the associated bicolored tree $T_A$ is a symbic tree.  Conversely, every symbic tree arises from a symmetric tropical matrix of symmetric tropical rank $2$, and the matrix is unique up to simultaneous tropical scaling of rows and columns. 
\end{theorem}
Simultaneous tropical scaling means adding a fixed real number $c$ to all entries in the $i$-th row and all entries in the $i$-column, hence adding $2c$ to the $(i,i)$-entry.  This does not change the associated symbic tree.

The next two lemmas will be used in the proof of Theorem~\ref{thm:symbic}.

\begin{figure}
\centering
\input{figures/3x3fans}
    \caption{The space of $3 \times 3$ symmetric tropical rank $2$ matrices, where $a, b \ge 0$.
    There are 9 top dimensional cells in the tropical hypersurface of the $3 \times 3$ symmetric determinant.  The symbic trees give a finer polyhedral structure with $12$ top dimensional cells.
    }
\label{fig:3x3fans}
\end{figure}
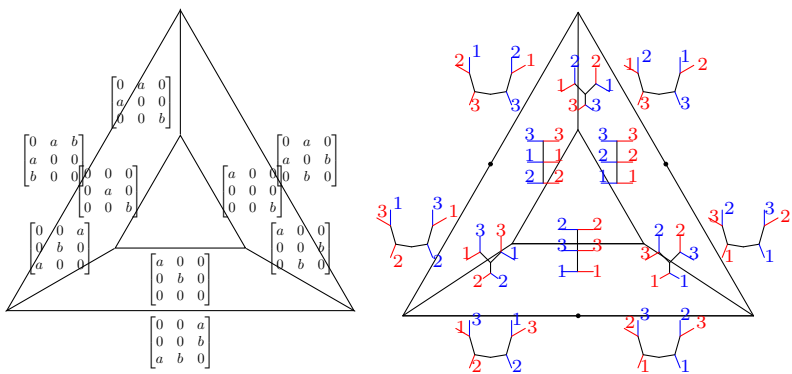



\begin{lemma}
\label{lem:transpose}
Let $A$ be an $n \times n$ matrix of tropical rank $2$ and $A^\top$ be its transpose.  Then the bicolored tree $T_{A^\top}$ is obtained from $T_A$  by swapping $i$ and $i'$ for every $i=1,\dots,n$.  In particular, if $A$ is symmetric, then $T_A$ is symmetric with respect to swapping the colors of every leaf.
\end{lemma}

\begin{proof}
Consider a bicolored tree $T$. It corresponds to a tropical rank 2 matrix, which is unique up to tropical scaling of rows and columns by~\cite{MarkwigYu}. We can construct one such matrix $A$ by assigning a point $O$ of the tree $T$ to the origin.  This determines the coordinates of the columns vectors of the matrix up to tropical scaling since the leaf labels of the tree encodes the directions of the edges. 
In particular, let $A$ be the matrix whose entry $A_{ij}$ is the length of path from $O$ to the point where the path from $O$ to $i$ and the path from $O$ to $j'$ diverge.  See Example~\ref{ex:4x4matrix} for a concrete example.  

We check that the Construction~\ref{con:tree-from-matrix} takes the the matrix $A$ back to $T$. By construction, a leaf $j'$ corresponds to an infinite edge in the $\mathbf{e}_j$ direction attached to the tropical line segment at that point. Consider the path from $O$ to $i$. Each edge along this path induces a split of some leaves $\{j_1', \dots, j_s'\}$ from $\{j_{s+1}', \ldots, j_n'\}$. Because this tree is balanced, this edge has direction $\mathbf{e}_{j_1} + \mathbf{e}_{j_2} + \cdots + \mathbf{e}_{j_s} = -\mathbf{e}_{j_{s+1}} - \mathbf{e}_{j_{s+2}} - \cdots - \mathbf{e}_{j_n}$. Without loss of generality, let $i$ be on the $j_1'$ side, then traversing this edge towards $i$ amounts to the usual addition of the vector $\ell(\mathbf{e}_{j_1} + \cdots +\mathbf{e}_{j_s})$, where $\ell$ is the length of that edge in $T_A$. Following this path, we have that the $j$-th coordinate of the point $i$ is exactly the length of the path from $O$ to the point where the path from $O$ to $i$ and the path from $O$ to $j'$ diverge. It follows that the tropical convex hull of the columns (or rows) of $A$ is the bicolored tree $T$ we started with, with the chosen point $O$ at the origin. 

By the description of the matrix $A$ above, changing the colors correspond to transposing the matrix.  
\end{proof}

\begin{lemma}\label{lem:submatrices}
In Construction~\ref{con:tree-from-matrix} taking submatrices corresponds to taking induced subtrees.
\end{lemma}
\begin{proof}
Consider a symmetric matrix $M$ of symmetric tropical rank $2$. Consider a tropical convex hull $P=\text{tconv}(M)$ and an $r\times r$-principal submatrix $M'$. Then, deleting rows outside $M'$ corresponds to a projection of $P$ into $\mathbb{R}^r/\mathbb{R}\mathbf{1}$ and deleting columns outside $M'$ corresponds to passing through subpolytope $P'$ of the image. Note that since $M$ has tropical rank $2$, its tropical convex hull is $1$-dimensional, i.e.\ a tree. Therefore, deleting rows and columns outside $M'$ can be understood as taking subtree of $\text{tconv}(M)$. As a bicolored tree of a symmetric matrix is constructed from a tropical convex hull of the matrix by attaching leaves, a corresponding bicolored tree of a principal submatrix is an induced subtree.

If we take a non-principal submatrix, we still achieve a tree with two-colored leaves, but the resulting tree will not be symmetric, nor necessarily satisfy the bicolored splits condition.
\end{proof}

\begin{proof}[Proof of Theorem~\ref{thm:symbic}]
Let $A$ be a symmetric tropical matrix of  symmetric tropical rank $2$, and let $T_A$ be the associated bicolored tree attained via Construction~\ref{con:tree-from-matrix}.
Then it already satisfies the first two conditions for being a symbic tree.  The third condition on symmetry with respect to swapping colors follows from Lemma~\ref{lem:transpose}.  For the fourth condition, suppose for sake of contradiction that there are three fixed edges meeting at a point.  Then the leaves beyond each of those branches only get swapped with each other, for the branch to remain fixed.  Choose pairs of leaves $(i,i')$, $(j,j')$, and $(k,k')$, one beyond each of the three branches.  Consider the $3 \times 3$ principal submatrix indexed by $i,j,k$. From Lemma~\ref{lem:submatrices}, as the corresponding subtree looks like the left figure in Figure~\ref{fig:3x3}, it is the bicolored tree of a matrix of the form \begin{equation}
\label{eqn:3x3matrix}
{\begin{pmatrix}
    a & 0 & 0 \\ 0 & b & 0 \\ 0 & 0 & c
\end{pmatrix}} \text{ where } a>0,b>0,c> 0.
\end{equation}   The matrix representing the tree is unique up to tropical scaling of rows and columns, and for each of them the minimum in the tropical determinant is attained at the two permutations corresponding to the same monomial in the symmetric $3\times 3$ determinant.  Thus
this $3 \times 3$ principal minor has symmetric tropical rank $>2$, contradicting the assumption that the original matrix has symmetric tropical rank $2$.  Thus the subset of fixed points, viewed as a tree, cannot contain a branch point with three edges, so it must be a path.

Conversely, given a symbic tree $T$, we can construct tropical rank $2$ matrix $A$ associated with it as in the proof of Lemma~\ref{lem:transpose}, choosing the point $O$ to be one of the fixed points under the color-swapping involution.  Then the matrix $A$ is symmetric by construction because the symbic tree is symmetric under the color changing involution and $O$ is fixed by the involution.
The uniqueness up to rescaling rows and columns  follows from the uniqueness in the general case~\cite{MarkwigYu}.

Since $A$ is a matrix associated to a bicolored tree, it must be tropical rank 2,  and since the fixed points of the tree under the color-swapping involution form a path, it does not have a principal submatrix of the form \eqref{eqn:3x3matrix}.  As this matrix is symmetric, it has symmetric tropical rank 2.
\end{proof}

\section{Simplicial fan structure}    
\label{sec:Cones}

In this section we will describe the set of symmetric tropical rank two matrices as a simplicial polyhedral fan.  The \textit{lineality space} of a polyhedral cone $C$ is the linear space consisting of all vectors $v$ such that $C+v = C$.  The lineality space of a polyhedral fan is the intersection of the linear spaces of the cones in the fan.

 \begin{lemma}
 \label{lem:lineality}
 The set of $n \times n$ symmetric tropical rank one matrices is the (ordinary) linear space      \(
\{X \odot X^\top \mid X \in \RR^n \}.
     \)
     It is contained in 
     the lineality space of the set of $n \times n$ matrices of symmetric tropical rank at most $r$ for any $r \geq 1$.
 \end{lemma}

 \begin{proof}
      From the definition with $2 \times 2$ minors, a symmetric matrix having \emph{symmetric tropical rank one} is equivalent to it having \emph{tropical rank one}.  Tropical rank one matrices have  the form $X \odot Y^\top$ where $X$ and $Y$ are column vectors. Symmetric matrices of tropical rank one can be transformed to the zero matrix by simultaneously tropically scaling rows and columns, so they are of the form $X \odot X^\top$.  These simultaneous tropical scaling operations do not change the symmetric tropical rank, so we have the result.
 \end{proof}

  As seen in the proof of Theorem~\ref{thm:symbic}, each symbic tree $T$ arises from a symmetric tropical rank two matrix $A_T$ whose $(i,j)$ entry is the distance from a chosen fixed vertex $O$ to the vertex where the path from $O$ to leaf $i$ and  the path from $O$ to leaf $j'$ diverge.   

We will now give another way to construct tropical matrices from symbic trees using distances between leaves in the tree.
Let $B_T$ be a matrix whose $(i,j)$ entry is the total length of internal edges along the path from $i$ to $j'$ in the symbic tree $T$.  See Example~\ref{ex:4x4matrix}.  

\begin{proposition}
\label{prop:generators}
For a combinatorial symbic tree $\cT$, let $C_\cT$ be the set of all symmetric tropical rank $2$ matrices with a symbic tree of type $\cT$. It is a polyhedral cone with the following description:
\begin{align*}
    C_\cT &= \{A_T \mid T \text{ is a symbic tree of type } \cT \} + L \\
    C_\cT &= \{-B_T \mid  T \text{ is a symbic tree of type } \cT \} + L
\end{align*}
where $A_T$ and $B_T$ are as described in the paragraph above, and
 $L = \{X \odot X^\top \mid X \in \RR^n\}$ is the linear space in the previous lemma.   
\end{proposition}
 
Modulo the linear space $L$, the dimension of the cone $C_\cT$ is the number of internal edges in $\cT$ up to the color swapping involution. 

\begin{proof}  The description of $C_\cT$
 in terms of $A_T$ follows from the proof of Theorem~\ref{thm:symbic}.  For the second description, consider the matrix $2A_T + B_T$.  Its $(i,j)$ entry is the sum of the distance from $i$ to $O$ and the distance from $j$ to $O$.  Hence $2A_T + B_T = D \odot D^\top$ where $D$ is the column vector whose $i$-th entry is the distance from $i$ to $O$.  Since $D \odot D^\top$ is in the lineality of $C_\cT$ by Lemma~\ref{lem:lineality}, we conclude that the set of matrices $A_T$ and the set of matrices $-B_T$ coincide modulo $L$.  The matrices $A_T$ and $B_T$ are matrices whose entries are linear forms in the internal edge lengths of the symbic trees, which must be positive.  This shows that the set $C_\cT$ has a linear parameterization by positive real numbers, so it is a polyhedral cone.
    \end{proof}


 \begin{example}
 \label{ex:4x4matrix}
 For $4 \times 4$ symmetric tropical rank two matrices, the lineality space consists of matrices of the form
 $$L =
 \begin{bmatrix}
2d & d+e & d+f & d+g \\
d+e & 2e & e+f & e+g \\
d+f & e+f & 2f & f+g \\
d+g & e+g & f+g & 2g
\end{bmatrix}
$$
where $d,e,f,g$ are any real numbers.
    For the symbic tree $\cT$ Figure~\ref{fig:4x4treeAMatrix} with the chosen fixed vertex $O$, the matrix $A_T$ and $B_T$ defined above are
$$A_T=\begin{bmatrix}
0 & a & 0 & 0 \\
a & 0 & 0 & 0 \\
0 & 0 & b & b \\
0 & 0 & b & b+c
\end{bmatrix}, ~~ B_T=\begin{bmatrix}
2a & 0 & a+b & a+b+c \\
0 & 2a & a+b & a+b+c \\
a+b & a+b & 0 & c \\
a+b+c & a+b+c & c & 0
\end{bmatrix}$$
By Proposition~\ref{prop:generators} the polyhedral cone $C_\cT$ has the parameterizations
\begin{align*}
C_\cT & = \{A_T + L \mid a,b,c,d \geq 0; d,e,f,g \in \RR\},\\
C_\cT &= \{-B_T + L \mid a,b,c,d \geq 0; d,e,f,g \in \RR\}.
\end{align*}
As in the proof of Proposition~\ref{prop:generators}, we can see that
\[
2A_T + B_T = \begin{bmatrix}
2a & 2a & a+b & a+b+c \\
2a & 2a & a+b & a+b+c \\
a+b & a+b & 2b & 2b+c \\
a+b+c & a+b+c & 2b+c & 2b+2c
\end{bmatrix} = \begin{bmatrix}
    a \\ a \\ b \\ b+c
\end{bmatrix} \odot \begin{bmatrix}
    a \\ a \\ b \\ b+c
\end{bmatrix}^\top
\]
where the entries $(a,a,b,b+c)$ are the distances from the chosen fixed vertex $O$ to the leaves $1,2,3,4$   respectively.

 \begin{figure}[H]
    \centering
    \input{figures/4x4treeAndAMatrix}
    \caption{A $4 + 4$ symbic tree.  A chosen fixed vertex $O$ can be used to parameterize the set of matrices corresponding to this tree.  See Example \ref{ex:4x4matrix}.}
    \label{fig:4x4treeAMatrix}
\end{figure}
\end{example}

\begin{proposition}
    \label{prop:simplicial}
The combinatorial symbic trees on $2n$ leaves form a simplicial complex.  The tropical variety of $n\times n$ symmetric tropical rank 2 matrices has a simplicial fan structure with this simplical complex.
\end{proposition}

\begin{proof}
For each $n \geq 2$, the set of combinatorial symbic trees on $2n$ leaves form a simplicial complex as follows.  
For a symbic trees, its collection of splits is invariant under the color changing involution.  Thus we will consider its splits up to symmetry. For any collection of symmetry classes of splits arising from a symbic tree $T$, any subcollection also arises from a symbic tree, namely the one obtained by contracting some of the edges of $T$.  This shows that the set of combinatorial symbic trees form cells of a simplicial complex whose vertices are the symmetry classes of splits.

The cones in Proposition~\ref{prop:generators} are compatible with this simplicial complex structure.  The cones corresponding to distinct combinatorial symbic trees are disjoint, as a symmetric tropical rank $2$ matrix uniquely determines a symbic tree.
For each combinatorial symbic tree $\cT$, it follows from the description $-B_T$ that, modulo the linear space $L$, the cone $C_{\cT}$ consists of nonnegative linear combinations of the matrices corresponding to symbic trees with a single symmetry class of splits. For two combinatorial symbic trees $\cT, \cT'$, let $\cT''$ be the combinatorial symbic tree whose set of splits is the intersection of splits of $\cT$ and splits of $\cT'$.  Then by the description of $-B_T$, we have $\overline{C_{\cT''}} = \overline{C_{\cT}} \cap \overline{C_{\cT'}}$. This shows that the tropical variety of symmetric tropical rank $2$ matrices is a simplicial fan whose cones are in bijection with symbic trees, and this bijection respects the simplicial complex structure of symbic trees described above.
\end{proof}

\begin{remark}\label{rem:fanstructurerefinement}
    The fan structure of $n + n$ symbic trees refines the fan structure of $n \times n$ symmetric tropical rank 2 matrices given by the $3\times 3$ minors.  In other words, if two symmetric tropical rank $2$ matrices correspond to the same combinatorial type of symbic tree, then for each of their $3\times 3$ tropical minors the minima are attained at the same places in both matrices, regardless of the lengths of internal edges.  A $3 \times 3$ minor corresponds to a $6$-leaf subtree with $3$ red leaves and $3$ blue leaves, and each term in the $3 \times 3$ minor corresponds to a matching between the red and blue leaves.  One can check, by enumerating all $6$-leaf trees with colored leaves and using the matrix $- B_T$ constructed above, that whether a matching achieves the maximum total edge length does not depend on the values of the edge lengths, and only depends on the combinatorial type of the subtree.
    \qed
\end{remark}


    

We provide a list of terminology for symbic trees that will be used for the remainder of the paper.
A \textit{cherry} of a symbic tree refers to a pair of leaves $i$ and $j'$ that are adjacent to the same internal vertex. The \textit{trunk} of a symbic tree is the path in the tree that is fixed under the color-swapping automorphism. 
A symbic tree whose trunk is a single vertex of degree 2, and whose internal vertices form a path, is called a \textit{caterpillar symbic tree}. See Figure~\ref{fig:brittlebranchexample} for examples. If the trunk of a tree is not a single vertex, the symbic tree has only degree $3$ or higher internal vertices assuming all internal edges have positive lengths. We say that a subtree of a symbic tree is a \textit{branch} if it is a maximal subtree that ``branches out'' from a vertex on the trunk, and its intersection with the trunk is only that vertex.
We say that a symbic tree with $n$ vertices is \textit{regular} if it has $n-1$ edges (up to symmetry) of nonzero length. If a symbic tree has at least one edge with its length zero, then we call it \textit{singular}. 
Note that in a regular symbic tree, each vertex on the trunk is connected to exactly two branches that are color-swapped opposites of each other, and possibly up to two edges on the trunk. Moreover, each branch is a trivalent subtree. 
For the remainder of the paper, we assume, unless explicitly specified otherwise, that we are exclusively working with regular symbic trees. 
For a regular symbic tree, we can perform an operation on this tree by contracting one edge and extending the other edge to obtain another regular symbic tree. We refer to this operation as a \textit{transition} from one regular symbic tree to another.

\section{Shellability}
\label{sec:Shellability}
We saw in Proposition~\ref{prop:simplicial} that the combinatorial symbic trees form a simplicial complex. 
We will now show that this simplicial complex is shellable, analogously to the result for bicolored trees in~\cite{MarkwigYu}. 
A simplicial complex is called \textit{shellable} if there exists a shelling order as defined below.  A topological realization of any shellable simplcial complex is homotopy equivalent to a wedge of sphere.  In particular, it has homology only in the top dimension. Thus shellability is a combinatorial technique for proving that a tropical variety has nice topology.

\begin{definition}
    A \textit{shelling} of a pure-dimensional simplicial complex is a total ordering $<$ on the maximal cells so that for any two maximal cells $C' < C$ there exists another maximal cell $C''$ such that
    \begin{itemize}
        \item $C' \cap C \subseteq C'' \cap C$
        \item $C'' < C$
        \item $C \setminus C'' = x$ where $x$ is a vertex of $C$. In other words, $C''$ and $C$ differ in exactly one vertex.
    \end{itemize}
\end{definition}

The goal is to show that the simplicial complex of (combinatorial) symbic trees is shellable.  The vertices of this simplicial complex correspond to symbic trees with only one internal edge up to color-swapping symmetry.  Each cell in the simplicial complex corresponds to a symbic tree, which can be identified with the collection of its symmetry classes of splits.

We will proceed by induction on the number of leaves.  
There are only two $2+2$ symbic trees, each with one split each. The simplicial complex consists of two disjoint vertices corresponding to the two splits, and it is shellable.
  For the induction step, let $n\geq 3$, and for each $(n-1) + (n-1)$ symbic tree, we want to order the ways to obtain $n + n$ symbic trees by attaching $n$ and $n'$ to it.  
  
  We can attach $n$ and $n'$  to either the interior of an existing edge or as a new cherry after extending the trunk on either end.  We will order all these ways of attaching $n$ and $n'$ to the same $(n-1) + (n-1)$ symbic tree. Choose an endpoint $v$ of the trunk of an $(n-1) + (n-1)$ symbic tree and consider a topological ordering of the edges starting from $v$.  That is, $e_1 < e_2$ when the path from $e_2$ to $v$ contains the edge $e_1$, and we extend this partial order arbitrarily to a total ordering. Then we get an ordering of the edges of the tree. Add a least element to that topological ordering corresponding to attaching $n$ to the trunk at $v$, and, if another trunk endpoint exists, a greatest element corresponding to attaching $n$ at the other trunk endpoint vertex. This is an ordering of the ways to attach $n$ and $n'$ to an $(n-1) + (n-1)$ symbic tree.

However, not every $n + n$ symbic tree arises this way.  In other words, removing $n$ and $n'$ from an $n + n$ symbic tree may not  result in an $(n-1) + (n-1)$ symbic tree. 
For example, the symbic tree in Figure~\ref{fig:brittlebranchexample}(a), removing the leaves $4$ and $4'$ result in a $3 + 3$ symbic tree, but in Figure~\ref{fig:brittlebranchexample}(b) removing $4$ and $4'$ does not result in a symbic tree because Figure~\ref{fig:brittlebranchexample}(d) results in splits containing only one color on each side.

\begin{definition}
Let $T$ be an $n + n$ symbic tree.
Assume that $n$ is a leaf of a cherry on a branch of $T$. 
If removing $n'$ results in a split where one side is uni-colored with at least two leaves, then this uni-colored part is called a \textit{brittle twig} of $T$.  
\end{definition}

See Figure~\ref{fig:brittlebranchexample} for examples.
Removing $n,n'$ from an $n + n$ symbic tree results in an $(n-1)+(n-1)$ symbic tree if and only if it does not contain a brittle twig.
A branch with a brittle twig itself forms a caterpillar subtree of $T$ with leaves labeled $(i_1,i_2,i_3 ,\dots, i_k)$ in this order, where $k \geq 2$ and $i_1$ is adjacent to $n'$ in $T$.  We will use this sequence $(i_1,i_2,i_3 ,\dots, i_k)$ to identify the brittle branch.

 \begin{figure}
    \centering
\input{figures/BrittleVsNotExample}
    \caption{
The trees (a) and (b) are symbic trees, and the trees (c) and (d) are the result of removing $4$ from (a) and (b) respectively.  The tree (c) is symbic, while (d) is not as the bicoloring condition is not satisfied.  Thus (a) is not \emph{brittle}, but (b) has a \emph{brittle twig} $(1,2)$. }
    \label{fig:brittlebranchexample}
\end{figure}

We now define a shelling order for the $n + n$ symbic trees as follows. Fix a topological ordering, as discussed earlier, of the edges and trunk endpoints of every $(n-1) + (n-1)$ symbic tree. For two symbic trees $T$ and $T'$ on $2n$ leaves, we have $T' < T$ when one of the following is true:
\begin{enumerate}
    \item Neither $T'$ nor $T$ have brittle twigs, and after deleting $n$ and $n'$ from both $T'$ and $T$, the resulting symbic trees $S'$ and $S$ have $S' < S$ in their ordering.
    
    \item Neither $T'$ nor $T$ have brittle twigs, and after deleting $n$ from both $T'$ and $T$, the resulting symbic trees $S'$ and $S$ have $S' = S$. Let $L$ be the chosen topological ordering on $S$. The neighbors of $n$ (in $T$ as well as $T'$) are elements of $L$; they are either a trunk vertex endpoint of $S$, or a vertex arising from subdividing an edge of $S$. We have $T' < T$ when the place where $n$ is attached in $T'$ is comes earlier $L$ than the place where $n$ is attached in $T$.

    \item $T'$ does not have a brittle twig but $T$ does.
    
    \item $T'$ and $T$ both have brittle twigs $(i_1,i_2, \dots ,i_k)$ and $(j_1,j_2, \dots, j_l)$ respectively, and the twig $(i_1,i_2,\dots, i_k)$ is lexicographically earlier than the twig $(j_1,j_2, \dots, j_l)$. For example, if trees $T''$, $T'$, and $T$ had the twigs $(1, 2, 3, 4)$, $(1,2,3,4,5)$, $(5,1)$ respectively, they would be ordered $T'' < T' < T$. \label{ordering:twiglexicographic}
    
    \item $T'$ and $T$ both have brittle twigs with identical ordered subsets, and after deleting the brittle twigs and replacing them with $n'$, 
    then the resulting symbic tree with fewer leaves corresponding to $T'$, comes earlier than the tree corresponding to $T$ in the defined ordering.
\end{enumerate}

\begin{theorem}
  The simplicial complex of $n + n$ symbic trees is shellable for each $n \geq 1$.  The above ordering is a shelling order.
\end{theorem}

\begin{proof}
    What we need to show is that, for any two symbic trees $T' < T$ on $2n$ leaves, there exists a symbic tree $T''$ on $2n$ leaves so that $T'' < T$ and $T''$ and $T$ differ by exactly one split $e$ that does not appear in $T'$. In particular, $T''$ should be obtained by a transition applied on $T$,
    and no split corresponding to edges that the transition was applied to should appear in $T'$.
    
    We consider case by case. As per the definition of the shelling order, let $x_{T}$ be the vertex adjacent to $n$ in $T$.
    
    \begin{enumerate}
        \item The symbic trees $T'$ and $T$ yield different symbic trees $S'$ and $S$ when deleting $n$. Thus, since $S$ and $S'$ are symbic trees on $2(n-1)$ leaves, we have a shelling order, and there is a $S''$ satisfying $S'' < S$, and $S''$ differs from $S$ by a single edge $e_S$. If $e_S$ corresponds to an edge of $T$ (in other words, $e_S$ is an edge that wasn't smoothed after deleting $n$), then this same split in $S''$ yields a new tree $T''$ with the desired properties by reattaching $n$ where it was in $T$.
        If $e_S$ was smoothed after deleting $n$, then since that edge $e_S$ was subdivided by $x_T$ and $n$ attached to form $T$, we know that $T'$ differs from $T$ by both splits that the non-leaf edges of $x_T$ correspond to. Thus, applying transition on either split to ``move'' the leaf $n$ yields a new symbic tree that differs from $T$ by one split $e$ that is not in $T'$. Since one of those transitions must move the leaf $n$ trunk-ward, choosing that transition lets us move the leaf to somewhere earlier in $L$ than it is in $T$, by case 2. Hence, we have $T'' < T$.
        \item The symbic trees $T'$ and $T$ have the same symbic tree $S' = S$ remaining after deleting $n$, and the place where $n$ was attached in $T'$ came before where $T$ attached it in the ordering. Let $e'$ and $e$ be the edges of $T'$ and $T$ respectively where $n$ is attached. We have a topological ordering $L$ on the edges of $S$, so let $v$ be the trunk vertex that started the ordering.
        Consider $x_{T}$ in $T$. It has two edges not leading to $n$, and one of them lies on the path from $x_T$ to $v$. If we apply a transition to ``move'' the leaf $n$ closer to $v$, we obtain a new tree $T''$ with $T'' < T$ because the edge that $x_{T''}$ lies on is closer to $v$ than the edge that $x_T$ lies, so is smaller in the ordering $L$. It remains to show that the edge we contract corresponds to a split not in $T'$. Notice that the splits that differ between $T$ and $T'$ all lie on the path between $x_T$ and $x_{T'}$ in $S$, including the two edges subdivided to form $x_T$ and $x_{T'}$. Furthermore, this path is the symmetric difference of the path from $x_T$ to $v$, and the path from $x_{T'}$ to $v$, as this is a tree and paths are unique. Therefore, the only way for the edge we contract to \emph{not} be on the $x_T$ to $x_{T'}$ path is if the path from $x_{T'}$ to $v$ contains the path from $x_T$ to $v$, but that would make $e <_L e'$ in the topological order.

        \item If $T$ has a brittle twig $i_1i_2 \dots i_k$ and $T'$ does not, then the edge connecting the twig to the rest of the tree gives the split with one partition $\{n, i_1', \dots, i_k'\}$. This split cannot exist in $T'$, as it can only appear in a twig. Applying a transition on this edge so that $i_k'$ is no longer part of the twig yields a tree $T''$ with a smaller twig than $T$, so $T'' < T$ by case (4).
        
        \item Let $T'$ and $T$ both have brittle twigs $i_1i_2 \dots i_k$ and $j_1j_2 \dots j_l$ respectively. If $T' < T$ because $i_1 i_2 \dots i_k$ is a prefix of $j_1j_2 \dots j_l$, then contracting the edge connecting the twig of $T$ to the rest of the tree and expanding $j_l$ away yields a tree $T''$ with a smaller twig, as per the previous case. This split is the entire brittle twig of $T$ with the rest of the tree, so $T'$ cannot contain this split.
        
        If $i_1 i_2 \dots i_k$ is not a prefix of $j_1j_2 \dots j_l$, let $m$ be the first index where the twigs differ, and so $i_m < j_m$ and $i_k = j_k$ for $1 \le k < m$. Then the every edge on the path in $T$ between $j_m'$ and where the leaf labelled $i_m'$ is gives a split that is not in $T'$. If this path is not strictly contained in the twig, then as before we can shorten the twig of $T$ to yield $T''$. If the path does not leave the twig, then the label $i_m$ must be found somewhere on the twig $j_1 \dots j_l$, and since $m$ is the first index the two twigs differ, it must appear at an index later than m. So in the brittle twig of $T$ there is a subsequence $j_mj_{m+1}\dots j_{m+p}$ where $j_{m+p} = i_m$, and since $j_{m+p} < j_m$ there must be a pair of adjacent leaves in this subsequence whose leaf labels are decreasing. Transposing those leaves gives a tree $T''$ whose twig is lexicographically earlier than $T$, as desired. This split contains $n$ and $j_m'$ on one side and $j_{m+p}' = i_m'$ on the other, and so is not in $T'$.
        
        \item Let the brittle twig that $T$ and $T'$ share be $i_1,i_2,\dots,i_k$.
        Then it is straightforward to see that removing and replacing the brittle twig $n, i_1,i_2,\dots,i_k$ with $n'$ (and applying the equivalent operation on the color-swapped branch on the other side), yields new symbic trees on $[n]\setminus \{i_1, \dots, i_k\}$. 
      Let $S$ and $S'$ be the symbic trees on $2(n-k)$ leaves obtained this way after relabelling the leaf sets $[n] \setminus \{i_1, \dots, i_k\}$ to $[n-k]$ in the canonical order-preserving way.

        Inductively, we have an ordering on $S$ and $S'$, and in particular, there is some $S''$ that differs from $S$ by one split not found in $S'$. This $S''$ corresponds to another tree $T''$ formed from the reverse process of relabelling the leaves back to $[n] \setminus \{n, i_1, \dots, i_k\}$, and replacing the leaf labelled $n'$ with the twig $(i_1i_2 \dots i_k)$, and $T''$ is the desired tree.\qedhere
    \end{enumerate}
\end{proof}

\section{Symbic tree enumeration}
\label{sec:enumeration}
We provide an exponential generating function for the number of regular combinatorial symbic trees. First we find an exponential generating function for the number of $n + n$ regular symbic trees with a one-vertex trunk, and an exponential generating function for the number of $n + n$ regular symbic trees with an $n$-vertex trunk.

\begin{lemma}\label{lem:1-vertex-trunk}
    The recursive formula for $a_n$ the number of $n + n$ regular symbic trees with a one-vertex trunk is $$a_n = \sum_{k=1}^{n-1} {n \choose k} a_k a_{n-k}, n \ge 3, \quad\text{with}\quad a_1 = a_2 = 1.$$
    The exponential generating function is 
    $$E_1(x) = \frac{1}{2} \left(1 -  \sqrt{1-4x+2x^2}\right).$$

    In particular, the first few terms of the sequence are $a_0 = 0$, $a_1 = 1$, $a_2 = 1$, $a_3 = 6$, $a_4 = 54$. It appears on the OEIS \cite{oeis} as Sequence A137591.
\end{lemma}

\begin{proof}
    We establish a combinatorial argument for the recursive formula. Note that a one-vertex trunk regular symbic tree consists of 2 color-swapped copies of a binary tree on $n$ leaves rooted at the trunk, where every non-leaf vertex has exactly 2 children, without ordering the children, and where any two leaves that share a parent must have opposite colors.

    Because of the color-swapping condition, we can restrict our attention to counting trees as above where the leaf $1$ is colored blue. The recurrence is as follows: There is exactly one such tree with one leaf, and one tree with $2$ leaves, as we need to satisfy the condition that leaves sharing a parent must have opposite colors, and the leaf $1$ is always blue. For an $n$-leaf tree, $n \ge 3$, we choose $k$ leaves to belong to one child, $1 \le k \le n-1$, and there are $a_k$ ways to arrange those children into a $k$-leaf tree, up to color-swapping. We are left with $n-k$ leaves to belong to the other child, and $a_{n-k}$ ways to arrange them into a tree up to color-swapping.

    However, only one of these two children can have the leaf $1$ and thus have their color predetermined by convention. We color the other tree as follows: If $1$ was in the $k$-leaf tree, we swap the colors of $(n-k)$-leaf tree so that the smallest leaf in that tree is red. Otherwise $1$ is in the $(n-k)$-leaf tree. It is blue there, and then we choose the colors of the $k$-leaf tree so that the smallest leaf in it is blue.

    This implicitly orders the children and gives the equivalence to the OEIS sequence A137591. 

    For the exponential generating function, we see that $E_1(x) = \sum_{n \ge 1} a_n\frac{x^n}{n!}$ satisfies \begin{align*}
        E_1(x)^2 &= \left(\sum_{i \ge 1} a_i\frac{x^i}{i!}\right)\left( \sum_{j \ge 1} a_j \frac{x^j}{j!}\right) \\&= \sum_{m \ge 2} \left( \sum_{i=1}^{m-1}\frac{a_i}{i!}\frac{a_{m-i}}{(m-i)!}\right)x^m \\&= x^2 + \sum_{m \ge 3} a_m \frac{x^m}{m!}.
    \end{align*}
    This gives us $E_1(x)^2 = E_1(x) - x - \frac12 x^2,$ and solving for $E_1(x)$ gives the formula for the exponential generating function as desired.
\end{proof}
\begin{lemma}\label{lem:n-vertex-trunk}
    The exponential generating function for the number of $n + n$ regular symbic trees with an $n$-vertex trunk is
    $$E_2(x) = \frac12 \left(1 + x + \frac1{1-x}\right).$$

    In particular, the first few terms of the sequences are $a_0 = 1, a_1 = 1, a_2 = 1, a_3 = 3, a_4 = 12.$
\end{lemma}

\begin{proof}
    An $n + n$ symbic tree with an $n$-vertex trunk is a permutation of $n$ up reversing the terms in the sequence. Thus, for $n \ge 2$, we have $\frac{n!}2$ many such symbic trees. For $n=1$ we have exactly one symbic tree, and we will need for $n=0$ to also have one tree.

    Thus, our exponential generating function has the form $$1 + x + \sum_{n\ge2} \frac{x^n}2 = \frac12 + \frac{x}2 + \sum_{n=0}^\infty \frac{x^n}2 = \frac12\left(1+x+\frac1{1-x}\right),$$ as desired.
\end{proof}

\begin{theorem}
    The exponential generating function for the number of $n + n$ regular symbic trees is $$E(x) = \frac34 - \frac{\sqrt{1-4x+2x^2}}4 + \frac1{1+\sqrt{1-4x+2x^2}}.$$
\end{theorem}
    The first few terms of the sequence are $a_0 = 1, a_1 = 1, a_2 = 2, a_3 = 12, a_4 = 111, a_5 = 1395.$

\begin{proof}
    We see that the number of $n + n$ regular symbic trees is the number of ways to partition $n$ into blocks, compose each block into a symbic tree with a one-vertex trunk, and then arrange these blocks along the trunk.
    Thus the composition formula for exponential generating functions (Theorem 5.1.4, \cite{stanley2011enumerative}) promises us that $E(x)$ is the composition of $E_2(x)$ with $E_1(x)$, and an easy algebraic manipulation shows that $E_2(E_1(x))$ has the desired form. 
\end{proof}

\begin{remark}
    By Remark~\ref{rem:fanstructurerefinement}, this provides an upper bound for the number of maximal cones of the symmetric tropical rank 2 matrices with the coarser polyhedral structure coming from $3\times 3$ minors.  From Figure~\ref{fig:3x3fans} the tropical variety from the $3 \times 3$ has $f$-vector $(1, 6, 9)$, but there are $12$ combinatorial types of $3\times3$ regular symbic trees.
\end{remark}

\section{The algebraic matroid of rank \texorpdfstring{$2$}{2} symmetric matrices}

\label{sec:Algebraic Matroids}
The algebraic matroid of an irreducible variety $V \subset K^n$ over an algebraically closed field $K$ is the matroid on the ground set $\{1,\dots,n\}$ whose rank function is given by the dimension of the projection of $V$ onto coordinate subspaces.  Equivalently, a subset $S \subset \{1,\dots,n\}$ is independent in the algebraic matroid if and only if the projection of $V$ onto the $S$ coordinates is dominant.  We can analogously define  algebraic matroids of tropical varieties, and tropicalization  preserves  algebraic matroids~\cite{Yu_algebraicMatroids}.

For the variety of rank $2$ matrices over an algebraically closed field, Bernstein gave a combinatorial description of the algebraic matroid by analyzing  the tropical variety~\cite{bernstein2017completion}.  We recall some useful lemmas.
 
\begin{lemma}\cite[Lemma 2.4]{bernstein2017completion}\label{lem:independentSetLemma}
Let $V\subset K^n$ be an irreducible affine variety of dimension $d$. Then $S\subset[n]$ is independent in the algebraic matroid $M(V)$ of $V$ if and only if $S$ is independent in $M(\text{span}(\sigma))$ for some $d$-dimensional cone $\sigma$ in $\text{trop}(V)$.
\end{lemma}

\begin{lemma}\cite[Lemma 2.5]{bernstein2017completion}\label{lem:basisTransitionLemma}
 Let $V\subset K^n$ be an irreducible $d$-dimensional variety . Let $\tau$ be a $(d-1)$-dimensional cone of $\text{trop}(V)$ and let $\sigma_1,\dots, \sigma_k$ be the $d$-dimensional cones in $\text{trop}(V)$ containing $\tau$. If $B\subset[n]$ is a basis of $M(\text{span}(\sigma_1))$ then $B$ is also a basis of $M(\text{span}(\sigma_i))$ for some $i\ne 1$.
\end{lemma}

Recall that a caterpillar tree is a tree with exactly two cherries, where a cherry is a pair of adjacent leaves.  
Using the second lemma above, Bernstein showed that it suffices to consider caterpillar trees only, that is, every basis of the algebraic matroid of the variety of rank 2 matrices appears as a basis for the cone corresponding to a caterpillar tree.  

With similar arguments, we show that for rank 2 symmetric matrices it suffices to consider only the symbic trees in which each branch is a caterpillar, meaning that it contains only one cherry.  
By algebraic matroid of a symbic tree, we mean the algebraic matroid of the linear span of the cone corresponding to the symbic tree.  

\begin{prop}\label{prop:caterpillar-basis}
The collection of bases in the algebraic matroid of the variety of rank two symmetric matrices is the union of the collections of bases of algebraic matroids of regular symbic trees.  Moreover, it suffices to take the union over only the symbic trees with caterpillar branches.
\end{prop}
\begin{proof}
The first statement follows from Lemma~\ref{lem:independentSetLemma} and Theorem~\ref{thm:symbic}.  For the second statement, consider a branch of the symbic tree and fix a cherry.  If the branch is not caterpillar, there is an internal edge on that branch that is not on the path between the trunk and the cherry.  In the left most figure in Figure~\ref{fig:branch-transition} the the fixed cherry is in the subtree $A$ and the red thick edge is not on the path from the trunk to the cherry. Contracting that edge gives a codimension one cone in the tropical variety, and both of the other maximal cones containing the codimension one cone have symbic trees with longer path from the trunk to the cherry, as shown on the figures on right.  By Lemma~\ref{lem:basisTransitionLemma}, every basis of the symbic tree on the left is a basis of one of the symbic trees on the right, so we no longer need to consider the tree on the left to find every basis of the whole matroid.  By repeating the process, we see that it suffices to consider only the symbic trees where all the interal edges on a branch are along the path from the trunk to the fixed cherry, that is, the branch is caterpillar.
\end{proof}

\begin{figure}
    \centering    \includegraphics[width=0.9\linewidth]{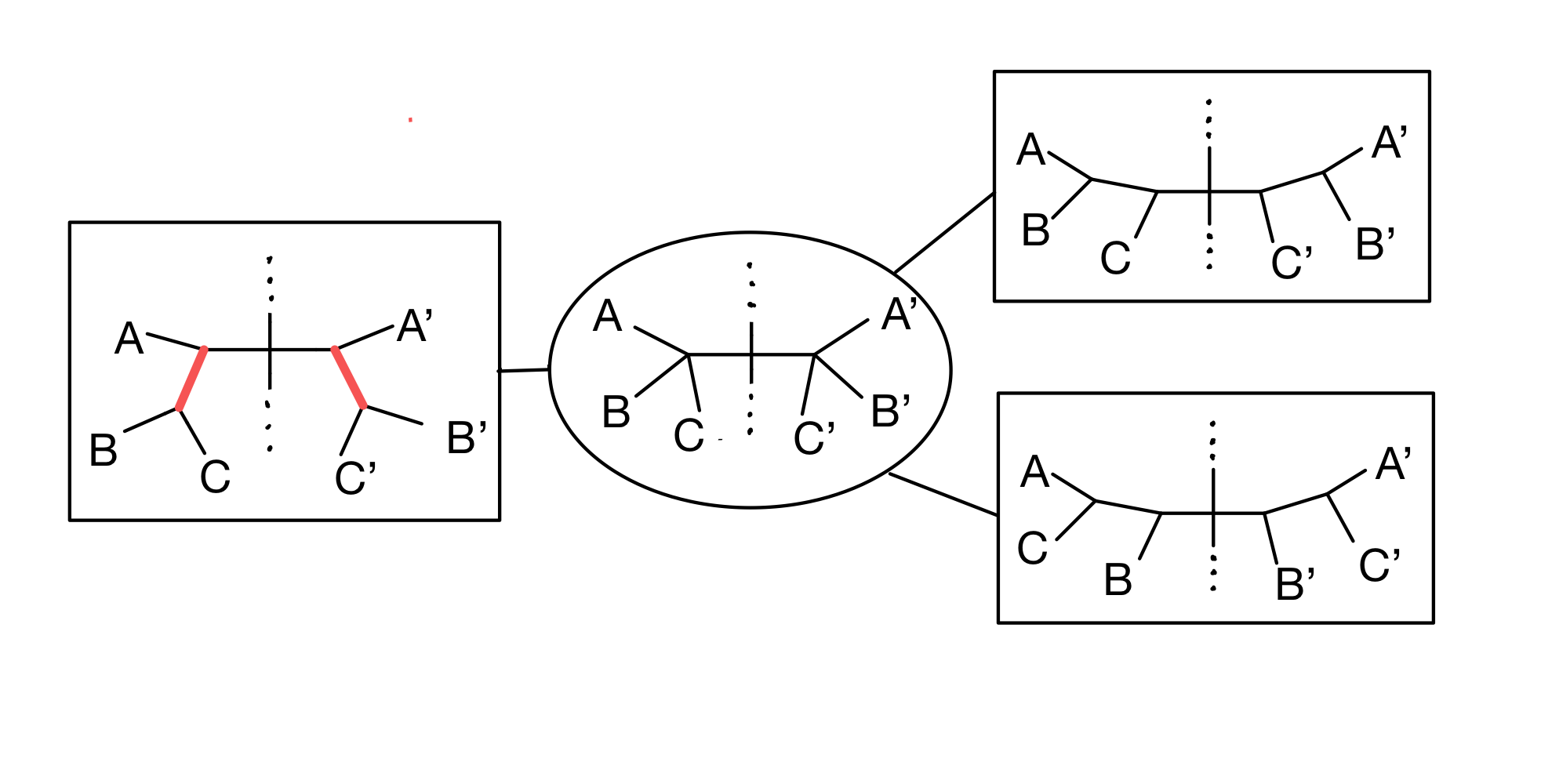}
    \caption{This picture shows that only caterpillar branches are needed for the algebraic matroid.  See  Proposition~\ref{prop:caterpillar-basis}. Here $A$, $B$, and $C$ are bicolored subtrees, not leaves.}
    \label{fig:branch-transition}
\end{figure}

We pose the following conjecture, similar to the result established in \cite[Lemma 3.3]{bernstein2017completion} for nonsymmetric matrices.
\begin{conj}
    Let $\mathcal{B}\subset E$ be a set of size $2n-1$. Then, $\mathcal{B}$ is a basis of a matroid of some symbic tree if and only if $\mathcal{B}$ is a basis of a matroid of some caterpillar symbic tree.
\end{conj}

 To prove this conjecture, we would wish to be able contract an edge on the trunk and transition to other symbic trees with shorter trunks, but it is not always possible, as seen in Figure~\ref{fig:trunk-transition}.

 \begin{figure}
     \centering    \includegraphics[width=0.7\linewidth]{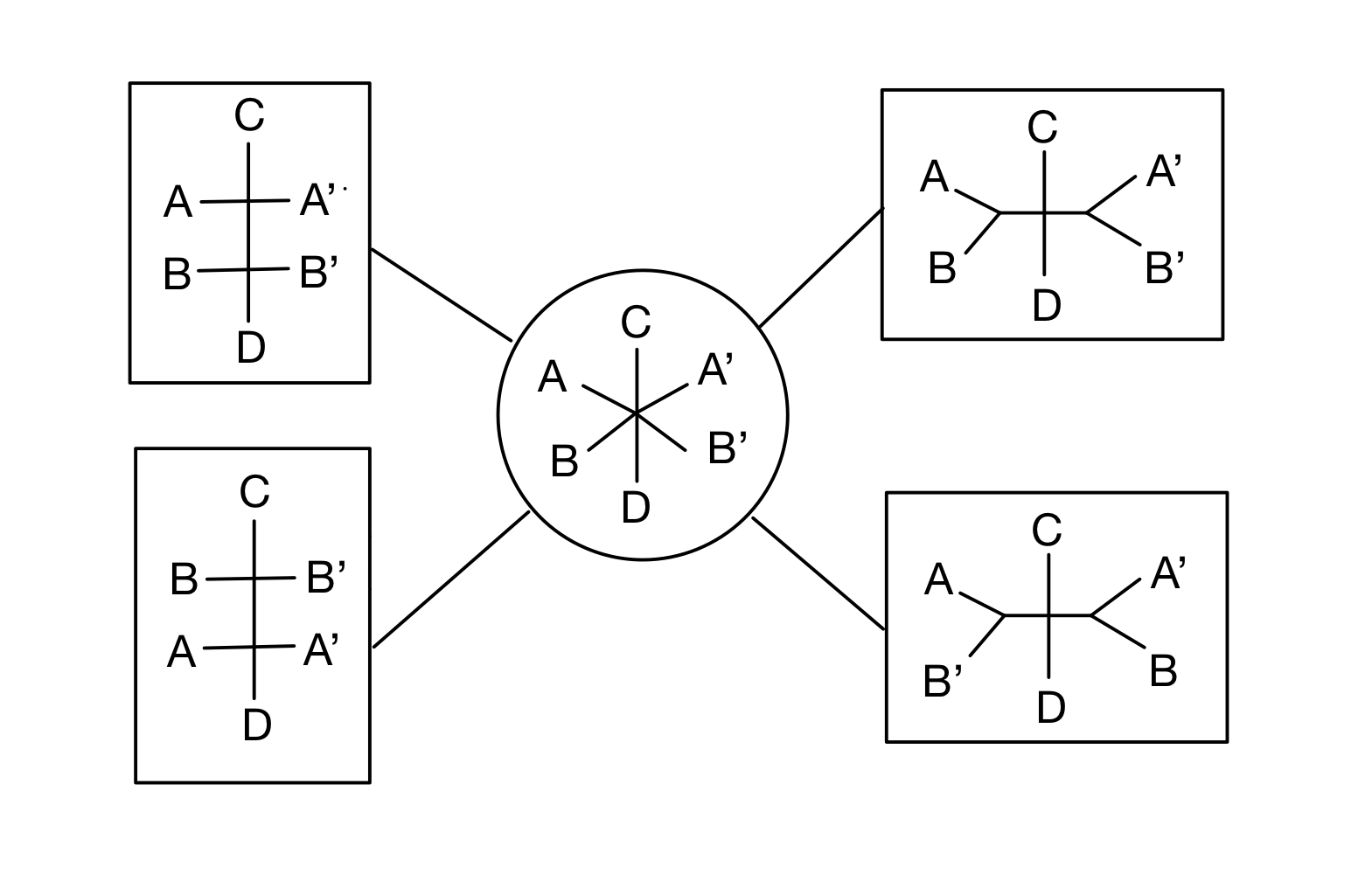}
     \caption{Starting with one of the symbic trees on the left, shrinking a trunk edge does not always allow a transition to symbic trees with shorter trunks. Here $A,B,C,$ and $D$ are bicolored subtrees, not just leaves.}
     \label{fig:trunk-transition}
 \end{figure}

Now we will see that the matroid associated with a symbic tree can be understood as a Cayley embedding of disjoint subsets of $\{\mathbf{e}_i + \mathbf{e}_j \mid i,j \in [n]\}$.
Recall that the Cayley embedding of point configurations $P_0,P_1,\dots,P_r \subset \ZZ^n$ is the following configuration in $\ZZ^n \times \ZZ^{r}$: \[(P_0 \times \{0\}) \cup (P_1 \times \{\mathbf{e}_1\}) \cup \cdots \cup (P_r \times \{\mathbf{e}_r\}).
\]
See \cite[Section 9.2]{de2010triangulations} for further details. 

For a symbic tree $T$, 
the cone of symmetric tropical rank $2$ matrices corresponding to $T$ is parameterized by the matrix $A_T$ described in the paragraph above Proposition~\ref{prop:generators}.  Now, instead of labeling each edge with  a parameter, we can re-parameterize by associating each node with a parameter and labeling each edge with the difference of the parameters for the end nodes, as in Figures~\ref{fig:relabel4x4} and~\ref{fig:relabelbig}.  The chosen fixed vertex $O$ is considered zero.  
It follows that the matroid of (the linear span of) the cone is given by the linear independence among columns of matrix with a Cayley embedding structure.  See Example~\ref{ex:Cayley}.

\begin{figure}
    \centering
    \input{figures/ex66figure}
      \caption{A symbic tree of $4$ leaves with its corresponding symmetric tropical rank $2$ matrix where $a \geq 0, b-a \geq 0$.}
    \label{fig:relabel4x4}
\end{figure}
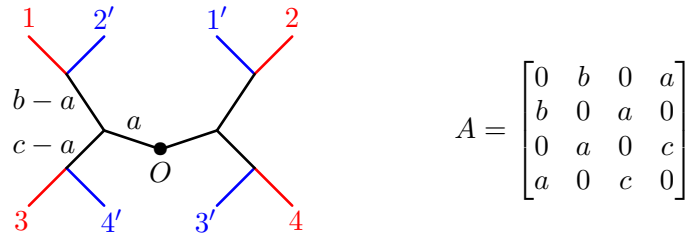

\pagebreak
\begin{example}
\label{ex:Cayley}
    The matroid associated with the tree in Figure \ref{fig:relabel4x4} is given by the following Cayley matrix. The parameters $d,e,f,g$ come from the lineality space as in Section~\ref{sec:Cones}.
\[  
\begin{blockarray}{ccccccccccc}
& 14' & 23' & 12' & 34' & 11' & 13' & 22' & 24' & 33' & 44' \\
\begin{block}{c[cc|c|c|cccccc]}
 a&1 & 1 & 0 & 0 & 0 & 0 & 0 & 0 & 0 & 0 \\
 b&0 & 0 & 1 & 0 & 0 & 0 & 0 & 0 & 0 & 0 \\
 c&0 & 0 & 0 & 1 & 0 & 0 & 0 & 0 & 0 & 0 \\
 \cmidrule(lr){2-3} \cmidrule(lr){4-4}  \cmidrule(lr){5-5}  \cmidrule(lr){6-11}
 d&1 & 0 & 1 & 0 & 2 & 1 & 0 & 0 & 0 & 0 \\
e& 0 & 1 & 1 & 0 & 0 & 0 & 2 & 1 & 0 & 0 \\
 f&0 & 1 & 0 & 1 & 0 & 1 & 0 & 0 & 2 & 0 \\
 g&1 & 0 & 0 & 1 & 0 & 0 & 0 & 1 & 0 & 2 \\
\end{block}
\end{blockarray}\]

\end{example}

\begin{figure}
    \centering
    \includegraphics[width=0.5\linewidth]{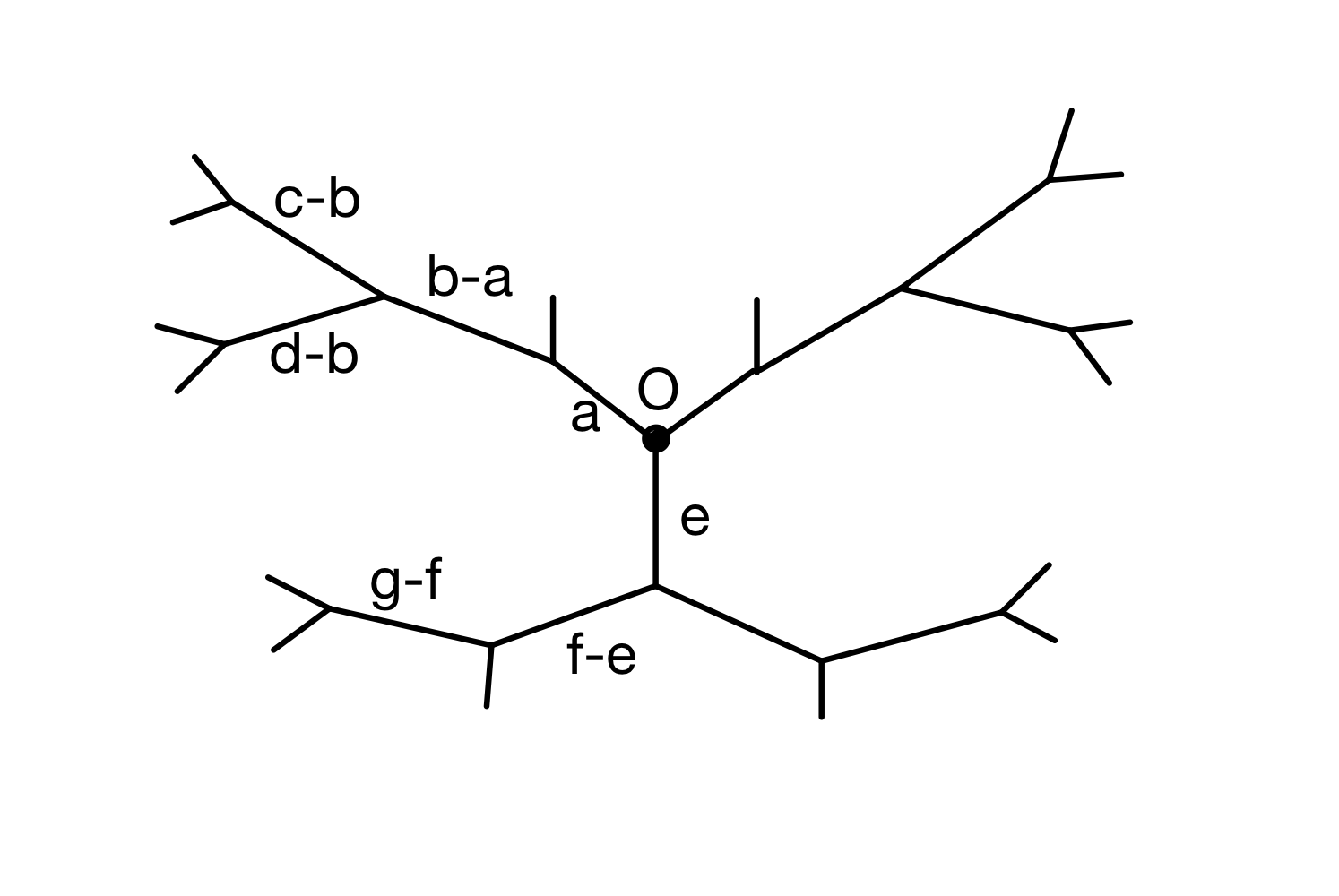}
    \caption{We can label the edges of a symbic tree such that the path distance from the chosen fixed vertex $O$ to each node is represented by a unique parameter.}
    \label{fig:relabelbig}
\end{figure}

The discussion so far is summarized in the following proposition: 

\begin{proposition}
    For an $n + n$ symbic tree $T$, its associated matroid is given by the affine independences in a Cayley sum of point configurations corresponding to nodes in $T$.  For each node $v$, associate the points $\mathbf{e}_i+\mathbf{e}_j$ such that the path from $O$ to $i$ and the path from $O$ to $j'$ diverge at the node $v$.
\end{proposition}

It remains an open problem to find simple combinatorial descriptions for the bases and circuits of the algebraic matroid of the variety of symmetric tropical rank $2$ matrices.

\section*{Acknowledgements}
We thank Skye Binegar, Cvetelina Hill, and Bo Lin for their contributions and Daniel Irving Bernstein for helpful discussions. We also thank Andrew Reimer-Berg for his insight on the recurrence in Lemma \ref{lem:1-vertex-trunk}. JY and MC were partially supported by NSF-DMS grant \#1855726.  We thank the referee for detailed comments which helped improve
the exposition.
\bibliography{main}
\bibliographystyle{alpha}

\end{document}

%% file: figures/01symbictrees.tex
\begin{tikzpicture}
\draw[line width=0.2mm,->] (0.5,0.5) -- (3.5,0.5) node[anchor=west]{$\mathbf{e}_2$};
\draw[line width=0.2mm,->] (0.5,0.5) -- (0.5,3.5) node[anchor=south]{$\mathbf{e}_3$};
\draw[line width=0.2mm,->] (0.5,0.5) -- (-1.5,-1.5) node[anchor=north]{$-\mathbf{e}_2-\mathbf{e}_3$};

\draw[line width=3pt] (0.5,0.5) -- (2.5,0.5);
\draw[line width=3pt] (0.5,0.5) -- (.5,2.5);
\draw[line width=3pt] (0.5,0.5) -- (-1,-1);

\draw[blue,line width=1.5pt] (-1,-1) -- (-1.7,-.7) node[blue,anchor=east]{$1'$};
\draw[red,line width=1.5pt] (-1,-1) -- (-.65,-1.65) node[red,anchor=north west]{$1$};

\draw[blue,line width=1.5pt] (2.5,.5) -- (3.2,1.2) node[blue,anchor=west]{$2'$};
\draw[red,line width=1.5pt] (2.5,.5) -- (3.2,-.2) node[red,anchor=west]{$2$};

\draw[blue,line width=1.5pt] (.5,2.5) -- (-.2,3.2) node[blue,anchor=east]{$3'$};
\draw[red,line width=1.5pt] (.5,2.5) -- (1.2,3.2) node[red,anchor=west]{$3$};

\draw[line width=0.2mm,->] (7.5,0.5) -- (10.5,0.5) node[anchor=west]{$\mathbf{e}_2$};
\draw[line width=0.2mm,->] (7.5,0.5) -- (7.5,3.5) node[anchor=south]{$\mathbf{e}_3$};
\draw[line width=0.2mm,->] (7.5,0.5) -- (5.5,-1.5) node[anchor=north]{$-\mathbf{e}_2-\mathbf{e}_3$};

\draw[line width=1.5pt] (7.5,0.5) -- (9.5,0.5);
\draw[line width=1.5pt] (7.5,0.5) -- (7.5,2.5);
\draw[line width=3pt] (7.5,0.5) -- (6,-1);

\draw[blue,line width=1.5pt] (6,-1) -- (5.3,-.7) node[blue,anchor=east]{$1'$};
\draw[red,line width=1.5pt] (6,-1) -- (6.35,-1.65) node[red,anchor=north west]{$1$};

\draw[blue,line width=1.5pt] (9.5,.5) -- (10.2,1.2) node[blue,anchor=west]{$2'$};
\draw[red,line width=1.5pt] (9.5,.5) -- (10.2,-.2) node[red,anchor=west]{$3$};

\draw[blue,line width=1.5pt] (7.5,2.5) -- (6.8,3.2) node[blue,anchor=east]{$3'$};
\draw[red,line width=1.5pt] (7.5,2.5) -- (8.2,3.2) node[red,anchor=west]{$2$};

\end{tikzpicture} 

%% file: figures/3x3fans.tex
\usetikzlibrary{shapes.geometric}
\usetikzlibrary{arrows}
\begin{center}
\begin{tikzpicture}[line cap=round,line join=round,x=1cm,y=1cm]

\node (inttri) [regular polygon, regular polygon sides=3, draw, minimum width=2cm] {};
\node (outtri) [regular polygon, regular polygon sides=3, draw, minimum width=5.33cm] {};

\draw (outtri.side 1) node[left, scale=0.5] {$\begin{bmatrix} 0 & a & b \\ a & 0 & 0 \\ b& 0 & 0\end{bmatrix}$};

\draw (outtri.side 2) node[below, scale=0.5] {$\begin{bmatrix} 0 & 0 & a \\ 0 & 0 & b\\a & b& 0\end{bmatrix}$};

\draw (outtri.side 3) node[right, scale=0.5] {$\begin{bmatrix} 0 & a & 0 \\ a & 0 & b \\ 0 & b & 0 \end{bmatrix}$};

\path (outtri.corner 1) -- (inttri.corner 1) [draw] node[midway, below left, scale=0.5] {$\begin{bmatrix} 0 & a & 0 \\ a & 0 & 0 \\ 0 & 0 & b\end{bmatrix}$};

\path (outtri.corner 2) -- (inttri.corner 2) [draw] node[midway, above, scale=0.5] {$\begin{bmatrix} 0 & 0 & a\\0 & b & 0\\a & 0 & 0\end{bmatrix}$};

\path (outtri.corner 3) -- (inttri.corner 3) [draw] node[midway, above, scale=0.5] {$\begin{bmatrix} a & 0 & 0\\0 & 0 & b\\0 & b & 0\end{bmatrix}$};

\draw (inttri.side 1) node[left, scale=0.5] {$\begin{bmatrix} 0 & 0 & 0 \\ 0 & a & 0 \\ 0 & 0 & b\end{bmatrix}$};

\draw (inttri.side 2) node[below, scale=0.5] {$\begin{bmatrix} a & 0 & 0 \\ 0 & b & 0 \\ 0 & 0 & 0\end{bmatrix}$};

\draw (inttri.side 3) node[right, scale=0.5] {$\begin{bmatrix} a & 0 & 0 \\ 0 & 0 & 0 \\ 0 & 0 & b\end{bmatrix}$};
\end{tikzpicture}
\hspace{0mm}
\begin{tikzpicture}[line cap=round,line join=round,>=triangle 45,x=1cm,y=1cm, scale=0.7]
\clip(-4.2,-7.1) rectangle (3.7,0);
\draw  (-1.59,-4.64)-- (0.91,-4.64);
\draw  (-3.66548831206455,-6.010486323429057)-- (2.99451168793545,-6.010486323429057);
\draw (-1.59,-4.64)-- (0.91,-4.64);
\draw (0.91,-4.64)-- (-0.34,-2.474936490538903);
\draw (-0.34,-2.474936490538903)-- (-1.59,-4.64);
\draw (-3.66548831206455,-6.010486323429057)-- (2.99451168793545,-6.010486323429057);
\draw (2.99451168793545,-6.010486323429057)-- (-0.33548831206454843,-0.24275713422469458);
\draw (-0.33548831206454843,-0.24275713422469458)-- (-3.66548831206455,-6.010486323429057);
\draw (-0.33548831206454843,-0.24275713422469458)-- (-0.34,-2.474936490538903);
\draw (-1.59,-4.64)-- (-3.66548831206455,-6.010486323429057);
\draw (0.91,-4.64)-- (2.99451168793545,-6.010486323429057);
\draw (-0.34891,-4.37086)-- (-0.34890712721425066,-5.170864257538686);
\draw [color=red] (-0.34891,-4.37086)-- (-0.04891,-4.37086);
\draw [color=red] (-0.3489085636071253,-4.770862128769343)-- (-0.04890856360712531,-4.770862128769343);
\draw [color=red] (-0.34890712721425066,-5.170864257538686)-- (-0.04890712721425067,-5.170864257538686);
\draw [color=blue] (-0.648909999992263,-4.370862154577845)-- (-0.34891,-4.37086);
\draw [color=blue] (-0.6489085635993883,-4.770864283347188)-- (-0.3489085636071253,-4.770862128769343);
\draw [color=blue] (-0.6489071272065137,-5.170866412116532)-- (-0.34890712721425066,-5.170864257538686);
\draw (-1,-2.68894)-- (-1,-3.48894);
\draw [color=red] (-1,-2.68894)-- (-0.7,-2.68894);
\draw [color=red] (-1,-3.08894)-- (-0.7,-3.08894);
\draw [color=red] (-1,-3.48894)-- (-0.7,-3.48894);
\draw [color=blue] (-1.3,-2.68894)-- (-1,-2.68894);
\draw [color=blue] (-1.3,-3.08894)-- (-1,-3.08894);
\draw [color=blue] (-1.3,-3.48894)-- (-1,-3.48894);
\draw (0.4,-2.68894)-- (0.4,-3.48894);
\draw [color=red] (0.4,-2.68894)-- (0.7,-2.68894);
\draw [color=red] (0.4,-3.08894)-- (0.7,-3.08894);
\draw [color=red] (0.4,-3.48894)-- (0.7,-3.48894);
\draw [color=blue] (0.1,-2.68894)-- (0.4,-2.68894);
\draw [color=blue] (0.1,-3.08894)-- (0.4,-3.08894);
\draw [color=blue] (0.1,-3.48894)-- (0.4,-3.48894);
\draw (-2,-1.8)-- (-1.7,-1.75);
\draw (-1.6,-1.4)-- (-1.7,-1.75);
\draw [color=blue] (-1.7,-1.75)-- (-1.6,-2);
\draw [color=blue] (-1.6,-1.1)-- (-1.6,-1.4);
\draw [color=red] (-1.6,-1.4)-- (-1.3401923788646686,-1.25);
\draw [color=blue] (-2.4,-1.1)-- (-2.4,-1.4);
\draw (-2.4,-1.4)-- (-2.3,-1.75);
\draw [color=red] (-2.4,-2)-- (-2.3,-1.75);
\draw (-2.3,-1.75)-- (-2,-1.8);
\draw [color=red] (-2.6598076211353314,-1.25)-- (-2.4,-1.4);
\draw (-3.5,-4.7)-- (-3.2,-4.65);
\draw (-3.1,-4.3)-- (-3.2,-4.65);
\draw [color=blue] (-3.2,-4.65)-- (-3.1,-4.9);
\draw [color=blue] (-3.1,-4)-- (-3.1,-4.3);
\draw [color=red] (-3.1,-4.3)-- (-2.8401923788646686,-4.15);
\draw [color=blue] (-3.9,-4)-- (-3.9,-4.3);
\draw (-3.9,-4.3)-- (-3.8,-4.65);
\draw [color=red] (-3.9,-4.9)-- (-3.8,-4.65);
\draw (-3.8,-4.65)-- (-3.5,-4.7);
\draw [color=red] (-4.159807621135326,-4.15)-- (-3.9,-4.3);
\draw (1.2,-6.8)-- (1.5,-6.75);
\draw (1.6,-6.4)-- (1.5,-6.75);
\draw [color=blue] (1.5,-6.75)-- (1.6,-7);
\draw [color=blue] (1.6,-6.1)-- (1.6,-6.4);
\draw [color=red] (1.6,-6.4)-- (1.8598076211353323,-6.25);
\draw [color=blue] (0.8,-6.1)-- (0.8,-6.4);
\draw (0.8,-6.4)-- (0.9,-6.75);
\draw [color=red] (0.8,-7)-- (0.9,-6.75);
\draw (0.9,-6.75)-- (1.2,-6.8);
\draw [color=red] (0.5401923788646763,-6.25)-- (0.8,-6.4);
\draw (-2,-6.8)-- (-1.7,-6.75);
\draw (-1.6,-6.4)-- (-1.7,-6.75);
\draw [color=blue] (-1.7,-6.75)-- (-1.6,-7);
\draw [color=blue] (-1.6,-6.1)-- (-1.6,-6.4);
\draw [color=red] (-1.6,-6.4)-- (-1.340192378864668,-6.25);
\draw [color=blue] (-2.4,-6.1)-- (-2.4,-6.4);
\draw (-2.4,-6.4)-- (-2.3,-6.75);
\draw [color=red] (-2.4,-7)-- (-2.3,-6.75);
\draw (-2.3,-6.75)-- (-2,-6.8);
\draw [color=red] (-2.6598076211353314,-6.25)-- (-2.4,-6.4);
\draw (2.8,-4.7)-- (3.1,-4.65);
\draw (3.2,-4.3)-- (3.1,-4.65);
\draw [color=blue] (3.1,-4.65)-- (3.2,-4.9);
\draw [color=blue] (3.2,-4)-- (3.2,-4.3);
\draw [color=red] (3.2,-4.3)-- (3.4598076211353317,-4.15);
\draw [color=blue] (2.4,-4)-- (2.4,-4.3);
\draw (2.4,-4.3)-- (2.5,-4.65);
\draw [color=red] (2.4,-4.9)-- (2.5,-4.65);
\draw (2.5,-4.65)-- (2.8,-4.7);
\draw [color=red] (2.14019237886467,-4.15)-- (2.4,-4.3);
\draw (1.2,-1.8)-- (1.5,-1.75);
\draw (1.6,-1.4)-- (1.5,-1.75);
\draw [color=blue] (1.5,-1.75)-- (1.6,-2);
\draw [color=blue] (1.6,-1.1)-- (1.6,-1.4);
\draw [color=red] (1.6,-1.4)-- (1.8598076211353316,-1.25);
\draw [color=blue] (0.8,-1.1)-- (0.8,-1.4);
\draw (0.8,-1.4)-- (0.9,-1.75);
\draw [color=red] (0.8,-2)-- (0.9,-1.75);
\draw (0.9,-1.75)-- (1.2,-1.8);
\draw [color=red] (0.540192378864669,-1.25)-- (0.8,-1.4);
\draw (-0.2,-1.8)-- (-0.2,-2);
\draw [color=blue] (-0.2,-2)-- (-0.06,-2.1);
\draw [color=red] (-0.2,-2)-- (-0.34,-2.1);
\draw (0,-1.6)-- (-0.2,-1.8);
\draw (-0.4,-1.6)-- (-0.2,-1.8);
\draw [color=red] (0,-1.3)-- (0,-1.6);
\draw [color=blue] (-0.4,-1.3)-- (-0.4,-1.6);
\draw [color=red] (-0.4,-1.6)-- (-0.6598076211353313,-1.75);
\draw [color=blue] (0,-1.6)-- (0.2598076211353316,-1.75);
\draw (-2,-5)-- (-2,-5.2);
\draw [color=blue] (-2,-5.2)-- (-1.86,-5.3);
\draw [color=red] (-2,-5.2)-- (-2.14,-5.3);
\draw (-1.8,-4.8)-- (-2,-5);
\draw (-2.2,-4.8)-- (-2,-5);
\draw [color=red] (-1.8,-4.5)-- (-1.8,-4.8);
\draw [color=blue] (-2.2,-4.5)-- (-2.2,-4.8);
\draw [color=red] (-2.2,-4.8)-- (-2.459807621135331,-4.95);
\draw [color=blue] (-1.8,-4.8)-- (-1.5401923788646685,-4.95);
\draw (1.4,-5)-- (1.4,-5.2);
\draw [color=blue] (1.4,-5.2)-- (1.58,-5.3);
\draw [color=red] (1.4,-5.2)-- (1.22,-5.3);
\draw (1.6,-4.8)-- (1.4,-5);
\draw (1.2,-4.8)-- (1.4,-5);
\draw [color=red] (1.6,-4.5)-- (1.6,-4.8);
\draw [color=blue] (1.2,-4.5)-- (1.2,-4.8);
\draw [color=red] (1.2,-4.8)-- (0.94019237886466,-4.95);
\draw [color=blue] (1.6,-4.8)-- (1.8598076211353316,-4.95);
\begin{tiny}
\draw [fill=black] (-2.000488312064549,-3.126621728826876) circle (1pt);
\draw [fill=black] (1.329511687935451,-3.126621728826876) circle (1pt);
\draw [fill=black] (-0.33548831206455,-6.010486323429057) circle (1pt);

\draw[color=red] (0.030076832656914257,-4.272320476808089) node {2};
\draw[color=blue] (-0.6403303165996036,-4.272320476808089) node {2};
\draw[color=red] (0.030076832656914257,-4.674410762349574) node {3};
\draw[color=red] (0.030076832656914257,-5.167150111018) node {1};
\draw[color=blue] (-0.6403303165996036,-4.674410762349574) node {3};
\draw[color=blue] (-0.6403303165996036,-5.167150111018) node {1};

\draw[color=red] (-0.6964359378379496,-2.5610990290385155) node {3};
\draw[color=blue] (-1.25749215022141,-2.5610990290385155) node {3};
\draw[color=red] (-0.6964359378379496,-2.9631893145799996) node {1};
\draw[color=red] (-0.6964359378379496,-3.355928663248427) node {2};
\draw[color=blue] (-1.25749215022141,-2.9631893145799996) node {1};
\draw[color=blue] (-1.25749215022141,-3.355928663248427) node {2};

\draw[color=red] (0.7062045931207009,-2.5610990290385155) node {3};
\draw[color=blue] (0.135797443864183,-2.5610990290385155) node {3};
\draw[color=red] (0.7062045931207009,-2.9631893145799996) node {2};
\draw[color=red] (0.7062045931207009,-3.355928663248427) node {1};
\draw[color=blue] (0.135797443864183,-2.9631893145799996) node {2};
\draw[color=blue] (0.135797443864183,-3.355928663248427) node {1};

\draw[color=blue] (-1.522178688270659,-1.9971825110514124) node {3};
\draw[color=blue] (-1.522178688270659,-0.999492571237865) node {2};
\draw[color=red] (-1.2322996452058712,-1.1491075612067895) node {1};
\draw[color=red] (-2.288955511861388,-1.9971825110514124) node {3};
\draw[color=blue] (-2.2609027012422147,-0.9714397606186917) node {1};
\draw[color=red] (-2.5881854917992335,-1.1491075612067895) node {2};

\draw[color=blue] (-2.999626714213771,-4.967920131080152) node {2};
\draw[color=blue] (-2.999626714213771,-3.870230191266605) node {3};
\draw[color=red] (-2.7003967342759254,-4.0198451812355293) node {1};
\draw[color=red] (-3.7570526009314418,-4.967920131080152) node {2};
\draw[color=blue] (-3.7570526009314418,-3.870230191266605) node {1};
\draw[color=red] (-4.056282580869288,-4.0198451812355293) node {3};

\draw[color=blue] (1.7038945329342372,-6.971880927518153) node {1};
\draw[color=blue] (1.7038945329342372,-5.974190987704606) node {2};
\draw[color=red] (2.0031245128720827,-6.12380597767353) node {3};
\draw[color=red] (0.9371177093435082,-6.971880927518153) node {1};
\draw[color=blue] (0.9371177093435082,-5.974190987704606) node {3};
\draw[color=red] (0.6472386662787205,-6.12380597767353) node {2};

\draw[color=blue] (-1.4941258776514859,-6.971880927518153) node {2};
\draw[color=blue] (-1.4941258776514859,-6.074190987704606) node {1};
\draw[color=red] (-1.2042468345866981,-6.22380597767353) node {3};
\draw[color=red] (-2.2609027012422147,-6.971880927518153) node {2};
\draw[color=blue] (-2.2609027012422147,-6.074190987704606) node {3};
\draw[color=red] (-2.5601326811800607,-6.22380597767353) node {1};

\draw[color=blue] (3.302904738227099,-4.867920131080152) node {1};
\draw[color=blue] (3.302904738227099,-3.970230191266605) node {3};
\draw[color=red] (3.6021347181649443,-4.1198451812355293) node {2};
\draw[color=red] (2.53612791463637,-4.867920131080152) node {1};
\draw[color=blue] (2.53612791463637,-3.970230191266605) node {2};
\draw[color=red] (2.246248871571582,-4.1198451812355293) node {3};

\draw[color=blue] (1.7038945329342372,-1.969129700432239) node {3};
\draw[color=blue] (1.7038945329342372,-1.0714397606186917) node {1};
\draw[color=red] (2.0311773234912556,-1.2210547505876162) node {2};
\draw[color=red] (0.9371177093435082,-1.969129700432239) node {3};
\draw[color=blue] (0.9651705199626812,-1.0714397606186917) node {2};
\draw[color=red] (0.6472386662787205,-1.2210547505876162) node {1};

\draw[color=blue] (0.01137495891079892,-2.0000428166550478) node {3};
\draw[color=red] (-0.43174939978870052,-2.0000428166550478) node {3};
\draw[color=red] (0.00488432764137562,-1.1678094349529053) node {2};
\draw[color=blue] (0.26671056008699043,-1.6166544048596788) node {1};
\draw[color=blue] (-0.3972059579001042,-1.1678094349529053) node {2};
\draw[color=red] (-0.6590321903457189,-1.6166544048596788) node {1};

\draw[color=blue] (-1.76,-5.35) node {2};
\draw[color=red] (-2.24,-5.35) node {2};
\draw[color=red] (-1.7998464888587548,-4.375180782411725) node {3};
\draw[color=blue] (-1.53802025641314,-4.824025752318498) node {1};
\draw[color=blue] (-2.192585837527177,-4.375180782411725) node {3};
\draw[color=red] (-2.454412069972792,-4.824025752318498) node {1};

\draw[color=blue] (1.68,-5.35) node {1};
\draw[color=red] (1.12,-5.35) node {1};
\draw[color=red] (1.6038945329342372,-4.375180782411725) node {2};
\draw[color=blue] (1.865720765379852,-4.824025752318498) node {3};
\draw[color=blue] (1.2018042473927574,-4.375180782411725) node {2};
\draw[color=red] (0.9399780149471427,-4.824025752318498) node {3};
\end{tiny}
\end{tikzpicture}
\end{center}

%% file: figures/4x4treeAndAMatrix.tex
\begin{tikzpicture}[line cap=round,line join=round,x=1cm,y=1cm, scale=0.5]

\draw [line width=1pt] (13,-1)-- (13,2);
\draw [line width=1pt] (13,2)-- (14.5,3);
\draw [line width=1pt] (13,2)-- (11.5,3);
\draw [line width=1pt,color=blue] (14.5,3)-- (13.5,4);
\draw [line width=1pt,color=blue] (11.5,3)-- (12.5,4);
\draw [line width=1pt,color=blue] (13,1)-- (14,0);
\draw [line width=1pt,color=blue] (13,-1)-- (14,-2);
\draw [line width=1pt,color=red] (13,1)-- (12,0);
\draw [line width=1pt,color=red] (11.5,3)-- (10.5,4);
\draw [line width=1pt,color=red] (14.5,3)-- (15.5,4);
\draw [line width=1pt,color=red] (13,-1)-- (12,-2);

\draw (13,2) node {$\bullet$};
\draw (13,2.7) node {$O$};

\draw (12,2.2) node {$a$};
\draw (14,2.2) node {$a$};
\draw (12.7,1.5) node {$b$};
\draw (12.7,-0.3) node {$c$};

\draw[color=blue] (13.5,4.4) node {$1'$};
\draw[color=blue] (12.5,4.4) node {$2'$};
\draw[color=blue] (14.3,-0.3) node {$3'$};
\draw[color=blue] (14.3,-2.3) node {$4'$};
\draw[color=red] (10.5,4.4) node {$1$};
\draw[color=red] (15.5,4.4) node {$2$};
\draw[color=red] (11.7,-0.3) node {$3$};
\draw[color=red] (11.7,-2.3) node {$4$};

\end{tikzpicture}

%% file: figures/BrittleVsNotExample.tex
\begin{tikzpicture}[line cap=round,line join=round,x=1cm,y=1cm, scale=0.5]

\draw [line width=1pt] (0,-0.5)-- (2,0);
\draw [line width=1pt] (0,-0.5)-- (-2,0);
\draw [line width=1pt] (2,0)-- (3,1);
\draw [line width=1pt] (-2,0)-- (-3,1);
\draw [line width=1pt] (3,1)-- (3,2);
\draw [line width=1pt] (-3,1)-- (-3,2);

\draw [line width=1pt, color=red] (-2,0) -- (-3,-1);
\draw [line width=1pt, color=blue] (-3,1) -- (-4, 0);
\draw [line width=1pt, color=red] (-3,2) -- (-4,3);
\draw [line width=1pt, color=blue] (-3,2) -- (-2,3);

\draw [line width=1pt, color=blue] (2,0) -- (3,-1);
\draw [line width=1pt, color=red] (3,1) -- (4, 0);
\draw [line width=1pt, color=blue] (3,2) -- (4,3);
\draw [line width=1pt, color=red] (3,2) -- (2,3);

\draw[color=blue] (-2,3.5) node {$4'$};
\draw[color=red] (-4,3.5) node {$1$};
\draw[color=blue] (-4.3, -0.1) node {$2'$};
\draw[color=red] (-3.2, -1.3) node {$3$};

\draw[color=red] (2,3.5) node {$4$};
\draw[color=blue] (4,3.5) node {$1'$};
\draw[color=red] (4.3, -0.1) node {$2$};
\draw[color=blue] (3.2, -1.3) node {$3'$};

\draw (0, -2) node {(a)};
\end{tikzpicture}
\begin{tikzpicture}[line cap=round,line join=round,x=1cm,y=1cm, scale=0.5]
    
\draw [line width=1pt] (0,-0.5)-- (2,0);
\draw [line width=1pt] (0,-0.5)-- (-2,0);
\draw [line width=1pt] (2,0)-- (3,1);
\draw [line width=1pt] (-2,0)-- (-3,1);
\draw [line width=1pt] (3,1)-- (3,2);
\draw [line width=1pt] (-3,1)-- (-3,2);

\draw [line width=1pt, color=blue] (-2,0) -- (-3,-1);
\draw [line width=1pt, color=red] (-3,1) -- (-4, 0);
\draw [line width=1pt, color=red] (-3,2) -- (-4,3);
\draw [line width=1pt, color=blue] (-3,2) -- (-2,3);

\draw [line width=1pt, color=red] (2,0) -- (3,-1);
\draw [line width=1pt, color=blue] (3,1) -- (4, 0);
\draw [line width=1pt, color=blue] (3,2) -- (4,3);
\draw [line width=1pt, color=red] (3,2) -- (2,3);

\draw[color=blue] (-2,3.5) node {$4'$};
\draw[color=red] (-4,3.5) node {$1$};
\draw[color=red] (-4.3, -0.1) node {$2$};
\draw[color=blue] (-3.2, -1.3) node {$3'$};

\draw[color=red] (2,3.5) node {$4$};
\draw[color=blue] (4,3.5) node {$1'$};
\draw[color=blue] (4.3, -0.1) node {$2'$};
\draw[color=red] (3.2, -1.3) node {$3$};

\draw (0, -2) node {(b)};
\end{tikzpicture}
\begin{tikzpicture}[line cap=round,line join=round,x=1cm,y=1cm, scale=0.5]

\draw [line width=1pt] (0,-0.5)-- (2,0);
\draw [line width=1pt] (0,-0.5)-- (-2,0);
\draw [line width=1pt] (2,0)-- (3,1);
\draw [line width=1pt] (-2,0)-- (-3,1);

\draw [line width=1pt, color=red] (-2,0) -- (-3,-1);
\draw [line width=1pt, color=blue] (-3,1) -- (-2, 2);
\draw [line width=1pt, color=red] (-3,1) -- (-4,2);

\draw [line width=1pt, color=blue] (2,0) -- (3,-1);
\draw [line width=1pt, color=red] (3,1) -- (2, 2);
\draw [line width=1pt, color=blue] (3,1) -- (4,2);

\draw[color=red] (-4,2.5) node {$1$};
\draw[color=blue] (-2, 2.5) node {$2'$};
\draw[color=red] (-3.2, -1.3) node {$3$};

\draw[color=blue] (4,2.5) node {$1'$};
\draw[color=red] (2, 2.5) node {$2$};
\draw[color=blue] (3.2, -1.3) node {$3'$};

\draw (0, -2) node {(c)};
\end{tikzpicture}
\begin{tikzpicture}[line cap=round,line join=round,x=1cm,y=1cm, scale=0.5]

\draw [line width=1pt] (0,-0.5)-- (2,0);
\draw [line width=1pt] (0,-0.5)-- (-2,0);
\draw [line width=1pt] (2,0)-- (3,1);
\draw [line width=1pt] (-2,0)-- (-3,1);

\draw [line width=1pt, color=blue] (-2,0) -- (-3,-1);
\draw [line width=1pt, color=red] (-3,1) -- (-2, 2);
\draw [line width=1pt, color=red] (-3,1) -- (-4,2);

\draw [line width=1pt, color=red] (2,0) -- (3,-1);
\draw [line width=1pt, color=blue] (3,1) -- (2, 2);
\draw [line width=1pt, color=blue] (3,1) -- (4,2);

\draw[color=red] (-4,2.5) node {$1$};
\draw[color=red] (-2, 2.5) node {$2$};
\draw[color=blue] (-3.2, -1.3) node {$3'$};

\draw[color=blue] (4,2.5) node {$1'$};
\draw[color=blue] (2, 2.5) node {$2'$};
\draw[color=red] (3.2, -1.3) node {$3$};

\draw (0, -2) node {(d)};
\end{tikzpicture}

%% file: figures/ex66figure.tex
\begin{tikzpicture}[line cap=round,line join=round,x=1cm,y=1cm, scale=0.5]

\draw [line width=1pt,color=blue] (15.5,3)-- (14.5,4);
\draw [line width=1pt,color=blue] (10.5,3)-- (11.5,4);
\draw [line width=1pt,color=blue] (15.5,.5)-- (14.5,-.5);
\draw [line width=1pt,color=red] (10.5,.5)-- (9.5,-.5);
\draw [line width=1pt,color=red] (10.5,3)-- (9.5,4);
\draw [line width=1pt,color=red] (15.5,3)-- (16.5,4);
\draw [line width=1pt,color=blue] (10.5,.5)-- (11.5,-.5);
\draw [line width=1pt,color=red] (15.5,.5)-- (16.5,-.5);

\draw [line width=1pt] (11.5,1.5)-- (13,1);
\draw [line width=1pt] (14.5,1.5)-- (13,1);
\draw [line width=1pt] (14.5,1.5)-- (15.5,3);
\draw [line width=1pt] (11.5,1.5)-- (10.5,3);
\draw [line width=1pt] (14.5,1.5)-- (15.5,.5);
\draw [line width=1pt] (11.5,1.5)-- (10.5,0.5);

\draw (12.3,1.7) node {$a$};
\draw (9.9,2.3) node {$b-a$};
\draw (9.9,1.1) node {$c-a$};

\draw[color=blue] (14.5,4.5) node {$1'$};
\draw[color=blue] (11.5,4.5) node {$2'$};
\draw[color=blue] (14.2,-.9) node {$3'$};
\draw[color=red] (16.6,-.9) node {$4$};
\draw[color=red] (9.5,4.5) node {$1$};
\draw[color=red] (16.5,4.5) node {$2$};
\draw[color=blue] (11.7,-.9) node {$4'$};
\draw[color=red] (9.3,-.9) node {$3$};

\draw (13,1) node {$\boldsymbol{\bullet}$};
\draw (13,.4) node {$O$};

\draw (24,1.5) node {$A=\begin{bmatrix}
0 & b& 0 & a \\
b & 0 & a & 0 \\
0 & a & 0 & c \\
a & 0 & c & 0
\end{bmatrix}$};

\end{tikzpicture}